\documentclass[a4paper,11pt]{amsart}
\usepackage[foot]{amsaddr}
\usepackage[utf8x]{inputenc}
\usepackage{dsfont}
\usepackage{amsmath,amsfonts,amssymb}
\usepackage{fullpage}
\usepackage{xcolor}
\usepackage{tikz}
\usetikzlibrary{calc}
\usetikzlibrary{arrows}
\usetikzlibrary{decorations.markings}
\usetikzlibrary{patterns,angles,quotes}
\usepackage{subcaption}
\usepackage{bm}

\newcommand{\lsf }{\phi}
\newcommand{\Om }{\Omega}

\newcommand{\divv}{\operatorname{div}}
\newcommand{\transp}{\mathsf{T}}

\newcommand{\Rd}{\R^d}
\newcommand{\tz}{t_0}
\newcommand{\po}{\partial\Om}

\newcommand{\dJ}{dJ}

\newcommand{\R }{\mathbb{R}}

\newcommand{\bfom}{\mathbf{\Om}}
\newcommand{\Sb}{\mathbf{S}}
\newcommand{\bx}{\bm{x}}

\newcommand{\bn}{\bm{n}}

\newcommand{\blsf}{\bm{\phi}}
\newcommand{\bta}{\bm{\theta}}
\newcommand{\bxi}{\bm{\xi}}
\newcommand{\Tt}{\bm{\Phi}_t}
\newcommand{\wblsf}{\bm{\psi}}
\newcommand{\wlsf}{\psi}

\newcommand{\I}{\mathcal{I}}
\newcommand{\K}{\mathcal{K}}
\newcommand{\mJ}{\mathcal{J}}

\newcommand{\hold}{\mathcal{D}}
\newcommand{\C}{\mathcal{C}}
\newcommand{\mM}{\mathcal{M}}
\newcommand{\mE}{\mathcal{E}}
\newcommand{\inte}{\operatorname{int}}
\newcommand{\pd}{\mathds{P}_\K(\hold)}
\newcommand{\ua}{u}
\newcommand{\ub}{v}
\newcommand{\pa}{p}
\newcommand{\pb}{q}
\newcommand{\xa}{\xi}
\newcommand{\xb}{\zeta}
\newcommand{\ya}{\mu}
\newcommand{\yb}{\eta}
\newcommand{\card}{{|\I|}}
\newcommand{\ck}{{\kappa}}
\newcommand{\test}{w}
\newcommand{\meas}{h}
\newcommand{\ke}{r}

\newtheorem{definition}{Definition}
\newtheorem{thm}{Theorem}
\newtheorem{lemma}{Lemma}

\newtheorem{proposition}{Proposition}
\newtheorem{assump}{Assumption}
\theoremstyle{remark}
\newtheorem{remark}{Remark}
\theoremstyle{definition}
\newtheorem{ex}{Example}
\title{Analysis and application of a lower envelope method for sharp-interface multiphase problems}
\author{Antoine Laurain}
\address{Instituto de Matem\'atica e Estat\'istica,
Universidade de S\~{a}o Paulo,
Rua do Mat\~{a}o, 1010,
CEP 05508-090 - S\~{a}o Paulo - SP (laurain@ime.usp.br)}
\date{\today}

\begin{document}
\begin{abstract}
We introduce and analyze a lower envelope method (LEM) for the tracking of interfaces motion in  multiphase problems.
The main idea of the method is to define the phases as the regions where the lower envelope of a set of functions coincides with exactly one of the  functions.
We show that a variety of complex lower-dimensional interfaces naturally appear in the process.
The phases evolution is then achieved by solving a set of transport equations.
In the first part of the paper, we show several theoretical properties, give conditions to obtain a well-posed behaviour, and show that the level set method is a particular case of the LEM.
In the second part, we propose a LEM-based numerical algorithm for multiphase shape optimization problems. 
We apply this algorithm to an inverse conductivity problem with three phases and present several numerical results.
\end{abstract}


\maketitle

\begin{center}
\end{center}

\section{Introduction}

The accurate modeling of multiple phases presenting sharp interfaces is highly relevant for physical phenomena and industrial processes.
Examples of such problems are the optimization of the distribution of several materials in order to minimize certain costs and mechanical criteria in  structural optimization, the monitoring of multiphase fluid flow in oil recovery system, the monitoring of sedimentation processes, multiphase inverse problems and dry foams.
On one hand,  the level set method (LSM) \cite{MR965860,MR1700751} introduced by Osher and Sethian has become a staple of sharp-interface modelling for two phases; see the recent review \cite{Gibou2018}.
On the other hand, the case of three or more interfaces presents additional challenges and is an active field of research.

A variety of level set-based methods have been proposed to handle multiphase algorithms in the literature.
The color level set method (CLSM) has been introduced in \cite{Vese2002} for image segmentation, see also  \cite{HLsegmentation,MR2126423,MR2033962}.
In this framework one can represent up to $2^n$ phases using $n=n_1+n_2$ functions, defining the phases as the sets where $n_1$ functions are positive, and $n_2$ functions are negative.
The multi-material level set-based method (MMLS), introduced in \cite{MR3283822}, uses $n$ level set functions to represent $n+1$ phases. 
The principle of the MMLS is similar to  the CLSM, but the phases are defined using different combinations of functions. 
The reconciled level set method (RLSM), also known as the coupled level set method \cite{MR2672120,MR3577907,zhang_chen_osher_2008}, has been introduced in \cite{MR1277282}
and is based on the diffusion of characteristic functions of each region.
We also mention the piecewise constant level set (PCLS) method \cite{MR2287609}, a projection method \cite{MR1914624},
and a smoothed interface approach using a signed distance function to enforce a fixed width of the transition layer in  \cite{MR3264217}.

These methods involve using multiple level set functions and occasionally some additional procedures such as projections to avoid the appearance of vacuums or overlaps. 
We observe that  level set-based methods have a fundamental limitation when it comes to capturing the motion of triple points and multiple junctions using smooth functions, which originates from the fact that the level sets of a smooth function are in most cases smooth, and consequently the nonsmoothness of the phases at a multiple junction must come from another mechanism.
In two dimensions for instance, the triple points appearing in the methods mentioned above usually have one angle equal to $\pi$ due to the smoothness of one of the phases at this junction point.
This observation suggests to explore other paradigms than level set approaches to track the motion of multiple junctions and phases.

Another issue is that many of these approaches involve a small diffuse interface, or a regularization parameter to smooth the level set functions. 
These regularization procedures introduce arbitrary parameters in the problem which need to be chosen ad hoc, may be unphysical, or need an asymptotic procedure to recover the sharp-interface configuration.

Among level set-based methods, the Voronoi Implicit Interface Method (VIIM)  \cite{MR2881484,Saye10052013} is an exception as  it  is able to capture the motion of multiple junctions and complex interfaces using only one function for an entire multiphase system. 
However, it also involves taking the limit of $\epsilon$-smoothed solutions as $\epsilon\to 0$  which makes its analysis challenging; see \cite{Laurain2017}.
Other methods not based on level sets include volume of fluid methods \cite{Noh}, front tracking methods \cite{MR1345029}, variational methods \cite{MR1408069},  SIMP \cite{MR3613464}, an alternating active-phase algorithm \cite{MR3183397}, and phase field models \cite{MR1740846}, where a diffusive layer with positive thickness models the interface.
The study of the sharp interface limit when  the thickness of the diffusive layer tends to zero is  an active field of research in the phase-field community;
see \cite{MR2375577,MR1632810}. 

In this paper we introduce a lower envelope method (LEM) for tracking the motion of interfaces in multiphase problems.
The LEM belongs to the class of implicit interface methods, but not to the class of level set methods, except for the particular case of two phases where it coincides with the LSM.
Regarding the issues discussed above, the LEM has the following advantages.
It does not involve any  regularization parameter or small diffuse interface, so the interfaces stay sharp at all times.
By construction, it precludes the appearance of vacuum and overlaps, and  naturally produces triple points
and other nonsmooth interfaces using smooth functions.
In particular, in two dimensions we can show that  the triple points have angles between $0$ and $\pi$ which can be explicitly computed using the functions involved in the method.  
Since no regularization or asymptotic procedure is required in the LEM, the analysis of the motion of multiple junctions and complex interfaces becomes much more tractable.

We give now a brief overview of the main ideas of the LEM.
Given a collection $\blsf$ of functions $\lsf_k$ in $\C^\infty(\R^d,\R)$, $k=0,\dots,d$, their lower envelope $E_{\blsf}$ is the supremum of the functions whose graph remains below  the union of the graphs of the functions in $\blsf$.
By construction,  $E_{\blsf}$ always coincides with one or more functions $\lsf_k$ at a point $x$.
On one hand, $E_{\blsf}$ is locally smooth at points where it coincides with exactly one function $\lsf_k$.
On the other hand, under certain natural conditions on $\blsf$ that will be discussed in details,  $E_{\blsf}$ is not smooth at points where it coincides with two or more functions $\lsf_k$, and the regions where  $E_{\blsf}$ coincides with exactly $d_0\leq d+1$ functions $\lsf_k$ are sets of dimensions $d- (d_0 - 1)$. 
The main idea of the LEM  is to exploit this key property of  $E_{\blsf}$ by defining  the phases as the sets where  $E_{\blsf}$ coincides with one of the functions $\lsf_k$. 
This naturally models a multiphase configuration with nonsmooth phases and a variety of lower-dimensional interfaces and multiple junctions such as triple points  in two dimensions, quadruple points and triple lines in three dimensions.
We show that the motion of these phases can be described by solving a set of transport equations,  generalizing the main idea of the level set method for nonsmooth domains described in \cite{MR3535238}.

In this paper we describe the LEM in the framework of multiphase optimization problems involving PDEs, considered as shape optimization problems \cite{MR2731611,SokZol92}.
In shape optimization, the derivative of the cost functional can be written in a weak form often called distributed shape derivative, which is a volume integral when the cost function is itself a volume integral, or in a strong form called boundary expression or Hadamard formula.
Boundary expressions are often computed for  domains which are at least $\C^1$, even though they can sometimes be computed for Lipschitz or polygonal domains, but this requires a careful analysis of the regularity of the solutions of the underlying PDEs; see \cite{laurain_2ndorder_shape_2019}. 
Distributed shape derivatives on the other hand are usually valid for domains with lower regularity such as curvilinear polygons, Lipschitz domains or even open sets.
Since the sets involved in multiphase optimization problems with at least three phases are usually curvilinear polygons, distributed shape derivatives are a key ingredient of the LEM.
Other advantages of shape derivatives in distributed form  are the higher accuracy  for numerical approximation; see  \cite{MR800331,Hiptmair2014}, and the fact that shape derivatives written in strong form are sometimes impractical for numerical purposes, as they may involve the  computation of jumps across interfaces; see the related discussions in \cite{MR3264217,MR3535238}.

In order to show the feasibility and efficiency of the LEM, we present an application to the inverse problem of electrical impedance tomography (EIT) with three phases.
In real-life problems, many applications of  EIT involve multiple phases and sharp interfaces.
The incorporation of  prior information about sharp interfaces  explicitly in the modeling of the problem is especially advantageous for inverse problems as they are characterized by incomplete data; see \cite{LIU2015108}.
Sharp-interface models for EIT with two phases have been studied in \cite{MR3620139, MR3723652,MR2536481,MR2886190,MR3535238,MR2052729}, but there are fewer references for three phases or more, we mention \cite{8425727} for a parametric level set method, and \cite{LIU2015108} for multi-phase flow monitoring.
In this paper we compute the distributed shape derivative for a general multiphase anisotropic EIT problem with piecewise smooth conductivity.
For the numerical experiments we consider the particular case of three phases and isotropic conductivity.

The paper is organized as follows.
In Section \ref{sec:intro} we define the lower envelope and the phases, study the properties of the phases distribution and give several examples.
In particular, we give a natural condition on the functions so that the phases distribution defines a partition of the domain without overlapping, which is a crucial property for the proper functioning of the algorithm. 
In Section \ref{sec:flow}, we define and discuss properties of weak and strong forms of shape derivatives in the multiphase setting.
In Section \ref{sec:HJ}, we demonstrate how the motion of phases, interfaces and multiple junctions can be tracked using transport equations, discuss the possibility of reducing the dimension of perturbation fields, introduce the LEM, and  show that the LSM \cite{MR965860} is a particular case of the LEM.
In Section \ref{sec:geo_lem} we study geometric properties of the LEM, in particular we compute the angles at a triple junction in two dimensions, and we verify that  multiple junctions evolve with the expected velocity.
In Section \ref{sec:EIT} we apply the LEM to a multiphase EIT problem and present several numerical experiments.

\section{Multiphase setting using a lower envelope function}\label{sec:intro}
In this section we introduce the multiphase setting based on a lower envelope approach.
The main task is to study the geometric properties of the phases and to give conditions on the lower envelope functions in order to avoid phases overlaps and obtain a partition of the domain.

Let $d\geq 2$ and $\hold\subset\R^d$ be open and bounded. 
Define the set of indices 
$$\K :=\{0,1,\dots ,\ck -1\}\subset\mathds{N},$$ 
where $\ck$ is the cardinal of $\K$, and  $\mathds{I}_k^r : = \{\I\subset\K\ |\ \card =r \text{ and } \I\ni k\}$. 
Let $\blsf = (\lsf_0,\lsf_1, \dots ,\lsf_{\ck -1})\in \C^\infty(\R^d,\R^{\ck})$.
\begin{definition}\label{def1}
The function
\begin{equation}\label{Edef}
E_{\blsf}(x) := \min_{k \in\K} \lsf_k(x)
\end{equation}
is called {\it lower envelope} of $\blsf$. 
We define the open sets
\begin{equation}
\label{def:omk} \Om_k(\blsf) := \inte\{x\in\hold\ |\ \lsf_k(x) = E_{\blsf}(x) \},\quad \mbox{ for } k\in\K, 
\end{equation}
or equivalently
\begin{equation}\label{def:omk:alt}
 \Om_k(\blsf) := \inte\left\{x\in\hold\ |\ \lsf_k(x) \leq \lsf_\ell(x),\ \forall \ell \in\K\setminus\{k\} \right\},\ \mbox{ for } k \in\K. 
\end{equation}
The sets $\Om_k(\blsf)$ are called ``phases''.
We denote by $\bfom(\blsf): = (\Om_0(\blsf),\dots,\Om_{\ck -1}(\blsf))$ the vector of phases $ \Om_k(\blsf)$.
\end{definition}
The following lemma describes several important properties of the phases $\Om_k(\blsf)$.
\begin{lemma}\label{lemma001}
For all $k \in\K$ we have  
\begin{equation}\label{eq1}
\left\{x\in\hold\ |\ \lsf_k(x) <\lsf_\ell(x),\  \forall \ell \in\K\setminus\{k\} \right\}\subset\Om_k(\blsf).
\end{equation}  
Moreover, for all $k \in\K$ we have 
\begin{equation}\label{eq2}
 \overline{\Om_k(\blsf)} = \left\{x\in\overline{\hold}\ |\ \lsf_k(x) \leq\lsf_\ell(x),\  \forall \ell \in\K\setminus\{k\} \right\} 
\end{equation}
and 
\begin{equation}\label{lemma4b}
\bigcup_{k\in\K} \overline{\Om_{k}(\blsf)} = \overline{\hold}. 
\end{equation}
\end{lemma}
\begin{proof}
The set $\left\{x\in\hold\ |\ \lsf_k(x) <\lsf_\ell(x),\  \forall \ell \in\K\setminus\{k\} \right\}$ is open since it is the preimage of an open set  under the vector-valued continuous function $$(\lsf_k - \lsf_0, \lsf_k - \lsf_1,\dots ,\lsf_k - \lsf_{k-1},\lsf_k - \lsf_{k+1},\dots,\lsf_k - \lsf_{\ck -1} )\in \C^\infty(\R^d,\R^{\ck -1}).$$
Since we clearly have the inclusion
$$\left\{x\in\hold\ |\ \lsf_k(x) <\lsf_l(x),\  \forall \ell \in\K\setminus\{k\} \right\}
\subset
\left\{x\in\hold\ |\ \lsf_k(x) \leq \lsf_l(x),\  \forall \ell \in\K\setminus\{k\} \right\},$$
and $\Om_k(\blsf)$ is by definition the largest open set included in $\left\{x\in\hold\ |\ \lsf_k(x) \leq \lsf_\ell(x),\, \forall \ell \in\K\setminus\{k\} \right\}$, \eqref{eq1} follows.

Now, taking the closure of both sets in \eqref{eq1}, we obtain
\begin{equation}\label{eq:5}
\left\{x\in\overline{\hold}\ |\ \lsf_k(x) \leq\lsf_\ell(x),\  \forall \ell \in\K\setminus\{k\} \right\} \subset \overline{\Om_k(\blsf)}. 
\end{equation}
Considering definition \eqref{def:omk:alt}, we also have
\begin{equation}\label{eq:6}
 \Om_k(\blsf) \subset \left\{x\in \hold\ |\ \lsf_k(x) \leq \lsf_\ell(x),\  \forall \ell \in\K\setminus\{k\} \right\}.
\end{equation}
Taking the closure of both sets in \eqref{eq:6} we obtain 
\begin{equation}\label{eq:7}
 \overline{\Om_k(\blsf)} \subset \left\{x\in\overline{\hold}\ |\ \lsf_k(x) \leq \lsf_\ell(x),\  \forall \ell \in\K\setminus\{k\} \right\}.
\end{equation}
Gathering \eqref{eq:5} and \eqref{eq:7} we obtain \eqref{eq2}.

The inclusion $\bigcup_{k\in\K} \overline{\Om_k(\blsf)} \subset \overline{\hold}$ in \eqref{lemma4b} is clear since $\overline{\Om_k(\blsf)}$ are subsets of $\overline{\hold}$. 
Conversely, take $x\in \overline{\hold}$ and $k \in\operatorname{argmin}_{\ell \in\K} \lsf_\ell(x)\neq \emptyset$.
Then, in view of \eqref{eq2} we have $x\in\overline{\Om_{k}(\blsf)}$ which proves that $\overline{\hold}\subset  \bigcup_{k\in\K} \overline{\Om_k(\blsf)}$.
Thus we obtain \eqref{lemma4b}.
\end{proof}
Without additional restrictions on $\blsf$, the sets  $\Om_k(\blsf)$ may overlap, which is an undesirable behaviour. 
This situation can be prevented by using the proper assumptions on $\blsf$ that we describe further.
We start with several definitions.
\begin{definition}\label{def2}
Let $\I = \{k_1,k_2, \dots , k_\card\}\subset\K$, where the cardinal $\card$ of $\I$ satisfies $2 \leq\card\leq \ck$ and $k_i < k_{i+1}$, $1\leq  i\leq \card -1$.
Define 
$$\widehat\blsf_\I := (\widehat\lsf_1, \widehat\lsf_2,\dots ,\widehat\lsf_{\card-1})\in \C^\infty(\R^d,\R^{\card-1})$$ 
with $\widehat\lsf_i : = \lsf_{k_1} - \lsf_{k_{i+1}}$ for $1\leq i\leq \card-1$.  
Define also 
\begin{align}
\label{eq:malpha0} \mM_{\I}(\blsf) &  := \{x\in \overline{\hold}\ |\ \widehat\blsf_\I(x) = 0 \}\\
\notag & =\{x\in\overline{\hold}\ |\ \lsf_{k_i}(x) = \lsf_{k_j}(x), \mbox{ for all }\ 1\leq  i,j\leq \card \mbox{ and } i\neq j\} ,\\
\label{eq:ei} \mE_\I(\blsf)  &  :=\bigcap_{k\in\I} \partial\Om_k(\blsf)  ,
\end{align}
where $\partial\Om_k(\blsf)$ denotes the boundary of $\Om_k(\blsf)$ in $\R^d$.
\end{definition}
The set $\mE_\I(\blsf)$ is the set of interfaces  shared by all the phases $\Om_k(\blsf)$ whose index $k$ belongs to $\I$. 
We will see that the set $\mM_\I(\blsf)$ is, roughly speaking, the union of $\mE_\I(\blsf)$ and some ``ghost'' interfaces that will be useful for the analysis; see Example \ref{example1}. 
Our aim is to avoid the situation where $\mM_{\I}(\blsf)$ is ``thick'', i.e. the dimension of $\mM_{\I}(\blsf)$ should be at most $d-(\card-1)$ when $\card\leq d$, otherwise differentiability issues would arise when defining the LEM. 
This property can be guaranteed by imposing the proper condition on $D\widehat\blsf_\I$.
\begin{lemma}\label{lemma1}
Let $\I \subset\K$, $2 \le \card\leq d$, and 
assume $D\widehat\blsf_\I(x)$ has maximal rank $\card-1$ for all $x\in \mM_\I(\blsf)$.
Then, $\mM_{\I}(\blsf)$ is a $\C^\infty$-manifold of dimension $d-(\card-1)$ and we have 
\begin{equation}\label{eq:01}
\mE_\I(\blsf) \subset \mM_{\I}(\blsf).
\end{equation} 
\end{lemma}
\begin{proof}
In view of definition \eqref{eq:malpha0}, we have
\begin{equation}\label{eq:003}
\mM_{\I}(\blsf) = \bigcap_{\I_2\subset\I,|\I_2|=2} \mM_{\I_2}(\blsf). 
\end{equation}
Then, for all $\I_2=\{k,\ell\}\subset\I$ with $k\neq \ell$, we have the property
\begin{equation}\label{eq:02}
\partial\Om_{k}(\blsf) \cap \partial\Om_{\ell}(\blsf) \subset \mM_{\I_2}(\blsf).
\end{equation}
Indeed, let $x\in \partial\Om_{k}(\blsf) \cap \partial\Om_{\ell}(\blsf)$, then in view of \eqref{eq2} we have in particular $\lsf_{k}(x) \leq \lsf_{\ell}(x)$ and  $\lsf_{\ell}(x) \leq \lsf_{k}(x)$. 
Thus $\lsf_{k}(x) = \lsf_{\ell}(x)$ which implies $x\in \mM_{\I_2}(\blsf)$.
Then, using \eqref{eq:003} we obtain \eqref{eq:01}.

Next, due to \eqref{eq:malpha0} we have $\mM_{\I}(\blsf) = \widehat\blsf_\I^{-1}(\{0\})\cap\overline{\hold}$ and since  by assumption $D\widehat\blsf_\I(x)$ has rank $\card-1$ for all $x\in \mM_\I(\blsf)$, then $0$ is a regular value of $\widehat\blsf_\I|_{\overline{\hold}}$. 
This shows that  $\mM_{\I}(\blsf) = \widehat\blsf_\I^{-1}(\{0\})\cap\overline{\hold}$ is a $\C^\infty$-manifold of dimension $d-(\card-1)$.
\end{proof}

Note that \eqref{eq:01} and \eqref{eq:02} are only inclusions in general, this is illustrated in Example \ref{example1}. 
Indeed, in view of \eqref{def:omk} it may happen that $x$ satisfies $\lsf_{j}(x) = \lsf_{k}(x)>\lsf_{\ell}(x) = E_{\blsf}(x)$ for some pairwise distinct indices $j, k, \ell$, which would imply $x\in\mM_{\{j, k\}}(\blsf)$ even though $x\notin\partial\Om_{j}(\blsf) \cap \partial\Om_{k}(\blsf)$. 
In this sense, $\mM_{\I}(\blsf)$ contains the ``ghost'' interfaces $\mM_{\I}(\blsf)\setminus \mE_{\I}(\blsf)$.

We now give a condition that guarantees the non-overlapping of the phases $\Om_k(\blsf)$. 
\begin{proposition}\label{lemma1b}
Let $\{k,\ell\}\subset\K$ with $k\neq \ell$.
If  $|\nabla (\lsf_k - \lsf_\ell)|> 0$  on $\mM_{\{k,\ell\}}(\blsf)$, then we have
\begin{equation}\label{prop2_lemma2}
\Om_{k}(\blsf) \cap \Om_{\ell}(\blsf) =\emptyset. 
\end{equation}
\end{proposition}
\begin{proof}
Assume that there exists $x\in \Om_{k}(\blsf) \cap \Om_{\ell}(\blsf)$. 
Since  $\Om_{k}(\blsf)$ and $\Om_{\ell}(\blsf)$ are open, there exists an open ball $B(x,r)$ of center $x$ and radius $r>0$ such that
$B(x,r)\subset \Om_{k}(\blsf) \cap \Om_{\ell}(\blsf)$.
Then, in view of \eqref{def:omk:alt} we have  for all $y\in B(x,r)$ that $\lsf_{k}(y)\leq \lsf_{\ell}(y)$ and $\lsf_{\ell}(y)\leq \lsf_{k}(y)$, which yields $\lsf_{k}(y) = \lsf_{\ell}(y)$.
Thus, we have $B(x,r)\subset\mM_{\{k, \ell\}}(\blsf)$ and consequently $D\widehat\blsf_{\{k,\ell\}}(y)= \nabla (\lsf_k - \lsf_\ell)(y) = 0$ for all $y\in B(x,r)$, which contradicts the hypothesis that $|\nabla (\lsf_k - \lsf_\ell)|> 0$ on $\mM_{\{k,\ell\}}(\blsf)$.
Thus we obtain \eqref{prop2_lemma2}. 
\end{proof}

The purpose of the next lemma is to give a characterization of the phase boundary 
$\partial\Om_k(\blsf)$ in terms of the sets $\mE_\I(\blsf)$.
This result is employed in Section \ref{sec:HJ} to model the motion of the interfaces $\partial\Om_k(\blsf)$ using $\blsf$. 
\begin{lemma}\label{lemma01}
For all $k\in\K$ we have
\begin{equation}\label{eq:109}
\hold\cap \bigcup_{\I\in\mathds{I}_k^2} \mE_\I(\blsf)  =  \hold\cap \bigcup_{\I\in\mathds{I}_k^r, r\ge 2} \mE_\I(\blsf) \subset \hold \cap \partial\Om_k(\blsf). 
\end{equation}
If in addition  $|D\widehat\blsf_{\I}| > 0$ on $\mM_{\I}(\blsf)$ for all $\I\in\mathds{I}_k^2$, then
\begin{equation}\label{eq:107}
\hold\cap \bigcup_{\I\in\mathds{I}_k^2} \mE_\I(\blsf)  = \hold\cap \bigcup_{\I\in\mathds{I}_k^r, r\ge 2} \mE_\I(\blsf) = \hold \cap \partial\Om_k(\blsf).
\end{equation}
\end{lemma}
\begin{proof}
Property \eqref{eq:109} is clear in view of definition \eqref{eq:ei} and the fact that $\mE_{\I^0}(\blsf)\subset \mE_\I(\blsf)$ if $\I\subset \I^0$.
Now suppose in addition that  $|D\widehat\blsf_{\I}| > 0$ on $\mM_{\I}(\blsf)$ for all $\I\in\mathds{I}_k^2$.
Then we have \eqref{prop2_lemma2} for any $\ell\in\K\setminus\{k\}$.
If $\partial\Om_k(\blsf)\cap\hold =\emptyset$, then \eqref{eq:107} is trivially satisfied, otherwise take $x\in \partial\Om_k(\blsf)\cap\hold$.
Since $\partial \Om_k(\blsf) = \overline{\Om_k(\blsf)}\setminus \Om_k(\blsf)$, it is not possible that $\lsf_k(x) <\lsf_\ell(x)$ for all $\ell \in\K\setminus\{k\}$, otherwise $x\in \Om_k(\blsf)$ in view of \eqref{eq1}.
Thus we must have $\lsf_k(x) = \lsf_\ell(x)$ for some $\ell \in\K\setminus\{k\}$.
In view of \eqref{eq2} this implies $x\in  \overline{\Om_\ell(\blsf)}\cap\hold$.
Then $x$ cannot belong to $\Om_\ell(\blsf)$, otherwise there would exist an open ball $B(x,r)\subset \Om_\ell(\blsf)$ with a non-empty intersection with $\Om_k(\blsf)$
and $\Om_{k}(\blsf) \cap \Om_{\ell}(\blsf)$ would not be empty, which would contradict \eqref{prop2_lemma2}.
Thus $x\in  \partial\Om_\ell(\blsf)\cap\hold$
and in turn $x\in\mE_{\{k,\ell\}}(\blsf)$, so this proves the other inclusion and yields \eqref{eq:107}.
\end{proof}

We now present a simple two-dimensional example to illustrate Lemma \ref{lemma1b}, Lemma \ref{lemma1} and Lemma \ref{lemma01}.
\begin{ex}\label{example1}
Let $\hold = (0,1)^2$,  $d = 2$,  $\K = \{0,1,2\}$,  $\I=\{k_1,k_2\} = \{0,1\}$, $\card=2$, and choose $\lsf_0\equiv 0$, $\lsf_1(x_1,x_2) = x_2-x_1$,  $\lsf_2(x_1,x_2) = 1- x_1-x_2$.
Then we have $\widehat\blsf_\I = (\widehat\lsf_1) =  (\lsf_{k_1} - \lsf_{k_2}) = (\lsf_0 - \lsf_1)= (- \lsf_1)$ and 
\begin{equation}
D\widehat\blsf_\I(x) 
= 
\begin{pmatrix}
1 & -1
\end{pmatrix}.
\end{equation}
Clearly, $D\widehat\blsf_\I(x)$ has rank $1$ for any $x\in\hold$, so we can apply Lemma  \ref{lemma1}, this shows that $\mM_{\I}(\blsf)$ is a a $\C^\infty$-manifold of dimension $d-(\card-1) = 1$.
An explicit calculation using \eqref{eq:malpha0} yields 
$$\mM_{\I}(\blsf) = \{x\in\overline{\hold}\ |\ \lsf_1(x) = 0 \}= \{x\in\overline{\hold}\ |\ x_1 = x_2 \},$$
so  $\mM_{\I}(\blsf)$ is a diagonal of the square $\hold$.
The lower envelope is $E_{\blsf} = \lsf_1\chi_{\Om_1(\blsf)} + \lsf_2\chi_{\Om_2(\blsf)}$ with
\begin{align*}
\Om_0(\blsf) & = \{x\in\hold\ |\ 0< \lsf_1(x) \mbox{ and }0< \lsf_2(x)\} = \{x\in\hold\ |\ x_1< x_2 \mbox{ and } x_2< 1-x_1\}, \\
\Om_1(\blsf) & = \{x\in\hold\ |\ \lsf_1(x)< 0 \mbox{ and }\lsf_1(x)< \lsf_2(x)\}= \{x\in\hold\ |\ x_2< x_1 \mbox{ and } 2x_2< 1\},\\
\Om_2(\blsf) & = \{x\in\hold\ |\ \lsf_2(x)< 0 \mbox{ and }\lsf_2(x)< \lsf_1(x)\}= \{x\in\hold\ |\ 1-x_1< x_2 \mbox{ and }1< 2x_2\}.
\end{align*}
Then, we compute 
$$\mE_{\{0,1\}}(\blsf) = \po_{0}(\blsf)\cap \po_{1}(\blsf) = \{x\in\overline{\hold}\ |\ x_1 = x_2 \mbox{ and } x_2\leq 1/2\}\subsetneq\mM_{\{0,1\}}(\blsf),$$
and we obtain similar characterizations for $\mE_{\{0,2\}}(\blsf)$ and $\mE_{\{1,2\}}(\blsf)$;
see Figure \ref{fig1} for an illustration of the geometry.

Finally, we can check that $|D\widehat\blsf_{\I}| > 0$ on $\mM_{\I}(\blsf)$ for all $\I\in\mathds{I}_0^2$, thus \eqref{eq:107} holds for $k=0$ according to Lemma \ref{lemma01}, and  \eqref{eq:107} becomes in this specific case
\begin{equation*}
\hold\cap (\mE_{\{0,1\}}(\blsf) \cup \mE_{\{0,2\}}(\blsf))  = \hold\cap (\mE_{\{0,1\}}(\blsf) \cup \mE_{\{0,2\}}(\blsf)\cup \mE_{\{0,1,2\}}(\blsf)) = \hold \cap \partial\Om_0(\blsf);
\end{equation*}
see Figure \ref{fig1}. Similar properties are obtained for $k=1$ and $k=2$ applying Lemma \ref{lemma01}.
\begin{figure}
\begin{center}
\begin{tikzpicture}[scale=7.0]
\draw[thick]  (0,0) --(1,0) --(1,1)--(0,1)--(0,0);
\draw[dashed, thick]  (0,0)--(0.5,0.5);
\draw[dashed, thick]  (0.5,0.5)--(1,1);
\draw[dash pattern=on \pgflinewidth off 1pt, thick]  (0,1)--(1,0);
\draw[dash pattern=on 3pt off 2pt on \the\pgflinewidth off 2pt, thick]  (0,0.5)--(1,0.5);
\node[text=black] at (0.5,0.9){$\hold$};
\node[text=black, rotate=45] at (0.72,0.8){$\mM_{\{0,1\}}(\blsf)$};
\node[text=black, rotate=-45] at (0.78,0.3){$\mM_{\{0,2\}}(\blsf)$};
\node[text=black] at (0.2,0.55){$\mM_{\{1,2\}}(\blsf)$};
\end{tikzpicture}
\hspace{0.5cm}
\begin{tikzpicture}[scale=7.0]
\draw[thick]  (0,0) --(1,0) --(1,1)--(0,1)--(0,0);
\draw[thick]  (0,0)--(0.5,0.5);
\draw[thick]  (0,1)--(0.5,0.5);
\draw[thick]  (0.5,0.5)--(1,0.5);
\node[text=black] at (0.8,0.8){$\Om_2(\lsf)$};
\node[text=black] at (0.8,0.2){$\Om_1(\lsf)$};
\node[text=black] at (0.2,0.5){$\Om_0(\lsf)$};
\node[text=black, rotate=45] at (0.28,0.2){$\mE_{\{0,1\}}(\blsf)$};
\node[text=black, rotate=-45] at (0.28,0.8){$\mE_{\{0,2\}}(\blsf)$};
\node[text=black] at (0.75,0.55){$\mE_{\{1,2\}}(\blsf)$};
\end{tikzpicture}
\caption{Illustration of the sets $\hold = (0,1)^2$, $\mM_{\I}(\blsf)$ (left) and $\Om_k(\blsf)$, $\mE_{\I}(\blsf)$ (right) for $\I = \{0,1\},\{0,2\},\{1,2\}$ from Example \ref{example1}, and of the triple point  $\mE_{\{0,1,2\}}(\blsf) = \{(\frac{1}{2},\frac{1}{2})\}$  from Example \ref{example2}.}\label{fig1}
\end{center}
\end{figure}
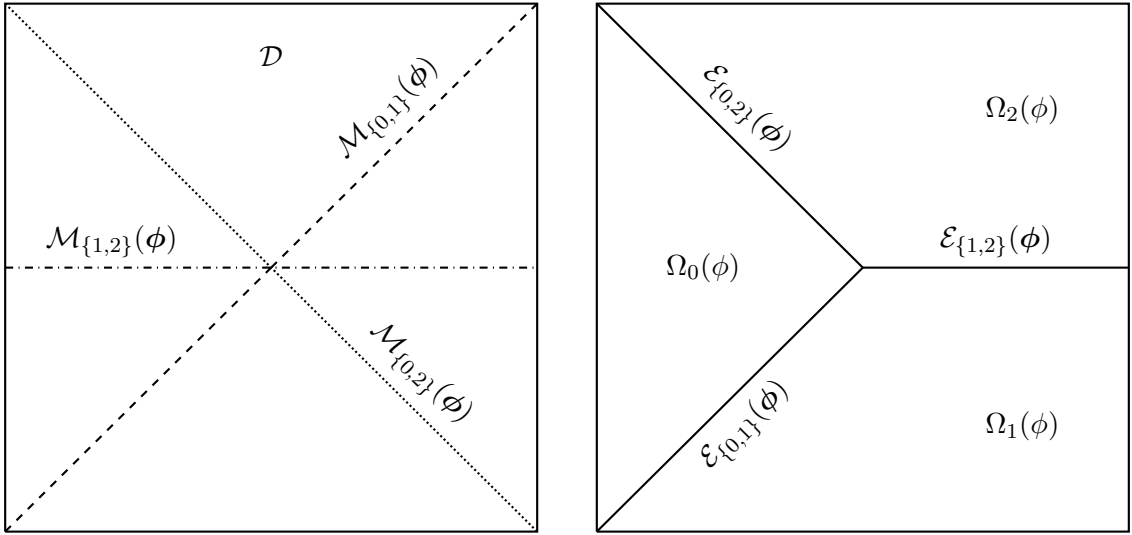 
\end{ex}

In Lemma \ref{lemma1} we have treated the case $\card\leq d$. 
Now we treat the degenerate case $\card \geq d+1$ where  $\mM_{\I}(\blsf)$ has zero dimension.
\begin{lemma}\label{lemma003}
Assume $\ck\geq d+1$  and  $\I  \subset \K$ with  $\card\geq d+1$.  
Suppose that $D\widehat\blsf_{\I}(x)$ has rank $d$ for all $x\in \mM_\I(\blsf)$.
Then, either $\mM_{\I}(\blsf) = \emptyset$ or $\mM_{\I}(\blsf)$ is a set of isolated points  and  we have
\begin{equation}\label{res05}
\mE_\I(\blsf) \subset \mM_{\I}(\blsf).
\end{equation} 
If in addition $\ck = d+1$  and $\I=\K$, then we also have 
\begin{equation}\label{res05b}
\mE_\I(\blsf) = \mM_{\I}(\blsf).
\end{equation} 
\end{lemma}
\begin{proof}
Assume $\mM_{\I}(\blsf)\neq\emptyset$ and let $x\in \mM_{\I}(\blsf)$. 
Thanks to the assumption that $D\widehat\blsf_{\I}(x)$ has rank $d$, there exists a subset $\I^0\subset\I$ with cardinal $|\I^0| = d+1$ such that the square matrix $D\widehat\blsf_{\I^0}(x)$ is invertible.
In view of Definition \ref{def2} we have $\mM_\I(\blsf)\subset\mM_{\I^0}(\blsf)$, and we also  have 
$\widehat\blsf_{\I^0}(x) = 0$ due to \eqref{eq:malpha0}.
Thus, we can apply the inverse function theorem, and there exists an open ball $B(x,\delta)$ for some $\delta>0$ such that $(\widehat\blsf_{\I^0})_{|B(x,\delta)}$ is a diffeomorphism. 
This yields  $\widehat\blsf_{\I^0}^{-1}(\{0\})\cap\hold = \{x\}$ and $\widehat\blsf_{\I^0}(y) \neq 0$ for $y\in B(x,\delta)\setminus\{x\}$, which shows that $x$ is an isolated zero of $\widehat\blsf_{\I^0}$, hence  $\mM_{\I^0}(\blsf)$ is a set of isolated points.
Since $\mM_\I(\blsf)\subset\mM_{\I^0}(\blsf)$, $\mM_\I(\blsf)$ is also a set of isolated points. 
Then we can prove that $\mE_\I(\blsf) \subset \mM_{\I}(\blsf)$ in a similar way as in Lemma~\ref{lemma1}.

Now we consider the particular case  $\ck = d+1$ and $\I=\K$.
In this case,  $D\widehat\blsf_{\I}(x)$ is a square matrix and the assumption  that $D\widehat\blsf_{\I}(x)$ has rank $d$ is equivalent to $D\widehat\blsf_{\I}(x)$ invertible. 
Let $x\in \mM_{\I}(\blsf)$, then we have by definition that $\lsf_k(x) = \lsf_\ell(x) \mbox{ for all } k,\ell\in\K$.
In view of \eqref{eq2}, this means that $x\in \overline{\Om_k(\blsf)}$ for all $k\in\K$.
We prove now that $x\in \partial\Om_k(\blsf)$ for all $k\in\K$.
Indeed, assume that $x\in \Om_k(\blsf)$ for some $k\in\K$.
In this case we prove that $x\notin \Om_\ell(\blsf)$ for all $\ell\in \K\setminus\{k\}$, otherwise there would exist some $\ell\in \K\setminus\{k\}$ such that $x\in \Om_k(\blsf)\cap\Om_\ell(\blsf)$.
Since this intersection is open, there would exist $B(x,r)\subset\Om_k(\blsf)\cap\Om_\ell(\blsf)$ with $r>0$, and we would have $\phi_k(y) = \phi_\ell(y)$ for all $y\in B(x,r)$ due to \eqref{eq2}.
This would imply that $D\widehat\blsf_{\I}(x)$ is not invertible which leads to a contradiction.
Thus, we must have $x\in \partial\Om_\ell(\blsf)$  for all $\ell\in \K\setminus\{k\}$.
If this was the case however, considering the assumption $x\in\Om_k(\blsf)$ we would have $B(x,r)\subset\Om_k(\blsf)$ and $B(x,r)\cap \Om_\ell(\blsf)\neq \emptyset$, for any $r>0$ sufficiently small, and in turn there would exist $y\in \Om_k(\blsf)\cap\Om_\ell(\blsf)\cap B(x,r)$ and  $D\widehat\blsf_{\I}(y)$ would again not be invertible.
Choosing $r$ sufficiently small, and considering that $\widehat\blsf_{\I}$ is smooth, this would contradict the hypothesis that $D\widehat\blsf_{\I}(x)$ be invertible.
Thus, the initial assumption $x\in \Om_k(\blsf)$ for some $k\in\K$ is not possible, and this proves that  $x\in \partial\Om_k(\blsf)$ for all $k\in\K$.
In this way we obtain $x\in\mE_\I(\blsf)$ and consequently  $\mM_{\I}(\blsf)\subset \mE_\I(\blsf)$, which yields \eqref{res05b}.
\end{proof}
\begin{definition}\label{def3}
When $\ck = d+1$, $\I=\K$  and the assumptions of Lemma \ref{lemma003} are satisfied, the elements of $\mE_\I(\blsf)$ are called $(d+1)$-tuple points.
In the particular case $d=2$, $(d+1)$-tuple points are called triple points following the standard denomination.
\end{definition}
\begin{ex}\label{example2}
Consider the same functions as in Example \ref{example1}, but now with $\card=d+1 = 3$ and $\I=\{k_1,k_2,k_3\} = \K = \{0,1,2\}$. 
We also have
$$\widehat\blsf_\I = (\widehat\lsf_1, \widehat\lsf_2) = (\lsf_{k_1} - \lsf_{k_2},\lsf_{k_1} - \lsf_{k_3})
= (\lsf_0 - \lsf_1,\lsf_0 - \lsf_2) = (- \lsf_1,- \lsf_2) $$
which yields
$$ D\widehat\blsf_\I(x) = 
\begin{pmatrix}
-1 & 1\\
-1 & -1
\end{pmatrix}.
$$
Clearly, $D\widehat\blsf_\I(x)$ is invertible for any $x\in\hold$ so the assumptions of Lemma \ref{lemma003} are satisfied, yielding $\mM_\I(\blsf) = \mE_\I(\blsf)$.
Furthermore, it is easy to see that  $\mM_\I(\blsf) = \mE_\I(\blsf) = \{(\frac{1}{2},\frac{1}{2})\}$ as in \eqref{res05b}.
\end{ex}
\begin{ex}\label{example3D}
Let $\hold = (0,1)^3$,  $d = 3$,  $\K = \{0,1,2,3\}$,  $\I=\{k_1,k_2,k_3\} = \{0,1,2\}$, $\card=3$, and choose $\lsf_0\equiv 0$, $\lsf_1(x_1,x_2,x_3) = x_2-x_1$,  $\lsf_2(x_1,x_2,x_3) = 1- x_1-x_2$,  $\lsf_2(x_1,x_2,x_3) = x_3 - 0.5$.
Then we have $\widehat\blsf_\I = (\widehat\lsf_1,\widehat\lsf_2) =  (\lsf_{k_1} - \lsf_{k_2},\lsf_{k_1} - \lsf_{k_3}) = (\lsf_0 - \lsf_1, \lsf_0 - \lsf_2)= (- \lsf_1,- \lsf_2)$ and 
\begin{equation}
D\widehat\blsf_\I(x) 
= 
\begin{pmatrix}
1 & -1 & 0\\
-1 & -1 & 0
\end{pmatrix}
\end{equation}
and $D\widehat\blsf_\I(x)$ has rank $2$ for any $x\in\hold$.
In view of Lemma \ref{lemma1}, $\mM_{\I}(\blsf)$ is a a $\C^\infty$-manifold of dimension $d-(\card-1) =1$.

Now for $\I=\K$ we would have
\begin{equation}
D\widehat\blsf_\I(x) 
= 
\begin{pmatrix}
1 & -1 & 0\\
-1 & -1 & 0\\
0 & 0 & 1
\end{pmatrix}
\end{equation}
which has rank $3$ for any $x\in\hold$, so we conclude in view of \eqref{res05b} that
$\mE_\I(\blsf) = \mM_{\I}(\blsf)$ is a set of isolated points.
An explicit calculation actually shows that $\mM_\I(\blsf) = \mE_\I(\blsf) = \{(\frac{1}{2},\frac{1}{2},\frac{1}{2})\}$.
\end{ex}

Gathering the results of this section, we have obtained a condition on $\blsf$ so that the phases $\Om_{k}(\blsf)$, $k\in\K$, form a partition of $\hold$, and that the dimension of the boundary of $\Om_{k}(\blsf)$ is at most $d-1$, i.e. the boundaries are not ``thick''. 
In fact, we have obtained a stronger result in this section since we have shown in Lemma \ref{lemma1} that the intersection of the boundaries of $\Om_{k}(\blsf)$ for $k\in\I$ has at most dimension $d-(\card-1)$, which allows to avoid degenerate situations.

We summarize these results in Theorem \ref{thm1}.
We first define partitions of $\hold$ indexed by $\K$.
\begin{definition}[$\K$-partitions of $\hold$]\label{def4}
Let $\mathds{P}$ denote the set of open subsets of $\hold\subset\R^d$.
For $\K =\{0,1,\dots ,\ck -1\}\subset\mathds{N}$, $\pd$ denotes the set of vector of domains  
$\bfom := (\Om_0,\dots,\Om_{\ck-1})$
with $\Om_k\in\mathds{P}$ for all $k\in\K$, $\Om_k \cap \Om_\ell =\emptyset$ for all $\{k,\ell\}\subset\K,\  k\neq \ell$ and $\bigcup_{k\in\K} \overline{\Om_k} = \overline{\hold}$. 
\end{definition}

\begin{thm}\label{thm1}
Let $\K =\{0,1,\dots ,\ck -1\}\subset\mathds{N}$, $k\in\K$,  $\blsf \in \C^\infty(\R^d,\R^{\ck})$ 
and $\Om_k(\blsf)$ defined as in Definition \ref{def1}.
Then, if $|D\widehat\blsf_{\I}| > 0$ on $\mM_{\I}(\blsf)$ for all $\I\in\mathds{I}_k^2$, we have
$$\bfom(\blsf) := (\Om_0(\blsf),\dots,\Om_{\ck-1}(\blsf))\in \pd$$
and the dimension of $\partial\Om_k(\blsf)$ is at most $d-1$.
\end{thm}
\begin{proof}
The fact that $\bfom(\blsf) \in \pd$  is an immediate consequence of Proposition \ref{lemma1b}, Lemma \ref{lemma001} and Definition \ref{def4}.
A direct application of Lemma \ref{lemma1} in the case $|\I| =2$ shows  that the dimension of $\partial\Om_k(\blsf)$ is at most $d-1$.
\end{proof}

\section{Multiphase shape optimization}\label{sec:flow}
We assume that $\hold$ is a  Lipschitz, simply connected, and piecewise $\C^1$ domain.
Denote by $\mathcal{S}$ the set of singular points of $\partial\hold$, then  the outward unit normal vector $\bn$  to $\hold$ is well-defined on $ \partial \hold\setminus \mathcal{S}$.
For $\ke\geq 1$ we define
\begin{align}
\C^{\ke}_c(\hold,\R^d) &:=\{\bta\in \C^{\ke}(\hold,\R^d)\ |\ \bta\text{ has compact support in } \hold\},\\
\C^{\ke}_{\partial \hold}(\overline{\hold},\R^d) &:=\{\bta\in \C^{\ke}( \overline{\hold},\R^d)\ |\ \bta\cdot \bn =0 \text{ on } \partial \hold\setminus \mathcal{S}\text{ and } \bta = 0\text { on }\mathcal{S}\}.
\end{align}
Consider a vector field $\bta\in \C^{1}_{\partial \hold}(\overline{\hold},\R^d)$ and the associated flow
$\Tt^{\bta}:\overline{\hold}\rightarrow \overline{\hold}$, $t\in [0,\tz]$ defined for each $x_0\in \overline{\hold}$ as $\Tt^{\bta}(x_0):=\bx(t)$, where $\bx:[0,\tz]\rightarrow \R^d$ is the solution to
\begin{align}\label{Vxt}
\begin{split}
\dot{\bx}(t)&= \bta(\bx(t))    \quad \text{ for } t\in [0,\tz],\quad  \bx(0) =x_0.
\end{split}
\end{align}
For $\Om\in \mathds{P}$, 
we consider the family of perturbed domains  
\begin{equation}\label{domain}
\Om_t := \Tt^{\bta}(\Omega). 
\end{equation}
In a similar way, for $\bfom\in\pd$ we define 
\begin{equation}\label{domain_vect}
\bfom_t := \Tt^{\bta}(\bfom) = (\Tt^{\bta}(\Omega_0),\dots,\Tt^{\bta}(\Omega_{\ck-1})). 
\end{equation}
For $\tz$ sufficiently small, it can be shown that $\Tt^{\bta}:\overline{\hold}\rightarrow\overline{\hold}$ is bijective and maps interior points onto interior points and boundary points onto boundary points; see \cite[Chapter 4, Section 5.1 and Remark 5.2]{MR2731611}.
A similar result holds if we take $\bta\in \C^{\ke}_c(\hold,\R^d)$ instead of $\bta\in \C^{1}_{\partial \hold}(\overline{\hold},\R^d)$.
This implies that $\bfom_t\in\pd$ for all $t\in [0,\tz]$.
When there is no ambiguity we will often write $\Tt$ for simplicity instead of $\Tt^{\bta}$ in the rest of the paper.

We are now ready to give the definition of shape differentiability.
\begin{definition}[Shape derivative]\label{def1b}
Let $J : \mathds{P} \rightarrow \R$ be a shape functional.
\begin{itemize}
\item[(i)] The Eulerian semiderivative of $J$ at $\Omega$ in direction $\bta \in  \C^{1}_{\partial \hold}(\overline{\hold},\R^d)$
is defined by, when the limit exists,
\begin{equation}
\dJ(\Omega)(\bta):= \lim_{t \searrow 0}\frac{J(\Omega_t)-J(\Omega)}{t}.
\end{equation}
\item[(ii)] $J$ is said to be \textit{shape differentiable} at $\Omega$ if it has a Eulerian semiderivative at $\Omega$ for all $\bta \in  \C^{1}_{\partial \hold}(\overline{\hold},\R^d)$ and the mapping
\begin{align*}
\dJ(\Omega):  \C^{1}_{\partial \hold}(\overline{\hold},\R^d) &  \to \R,\; \bta     \mapsto \dJ(\Omega)(\bta)
\end{align*}
is linear and continuous, in which case $\dJ(\Omega)(\bta)$ is called the \textit{shape derivative}  of $J$ at $\Omega$ in direction $\bta \in  \C^{1}_{\partial \hold}(\overline{\hold},\R^d)$.
\end{itemize}
\end{definition}
For a multiphase functional $\mJ : \pd\rightarrow \R$, we define the Eulerian shape derivative $d\mJ(\bfom)(\bta)$ in a similar way as
\begin{equation}
d\mJ(\bfom)(\bta):= \lim_{t \searrow 0}\frac{\mJ(\bfom_t)-\mJ(\bfom)}{t}.
\end{equation}
For transformations $\bm{\Phi}_t$ satisfying $\bm{\Phi}_t(\Omega)=\Omega$ for all $t\in [0,\tz]$, the shape derivative clearly vanishes.
When $\Om$ is at least $\C^1$, this leads to the following structure theorem proved by Zol\'esio in \cite{zolesio1979identification}, see also \cite{MR2731611,SokZol92}.
\begin{thm}[Structure theorem]\label{thm:structure_theorem}
Let $\Om\in\mathds{P}$ be of class $\C^{\ke+1}$, $\ke\geq 0$.
Suppose $J$ is shape differentiable at $\Om$ and $\dJ(\Omega )$ is continuous for the $\C^\ke(\hold,\R^d)$-topology.
Then, there exists a linear and continuous functional $L: \C^\ke(\partial \Omega )\rightarrow \R$ such that for all $\bta\in  \C^\ke_{\partial\hold}(\overline{\hold},\Rd)$,
\begin{equation}
\label{volume}
\dJ(\Omega)(\bta)= L(\bta_{|\po }\cdot \bn).
\end{equation}
\end{thm}
\begin{proof}
See \cite[pp. 480-481]{MR2731611}.
\end{proof}
Despite its usefulness in the case of two phases, Theorem  \ref{thm:structure_theorem} is not relevant in the multiphase context, where usually all the entries $\Om_k$ of the vector $\bfom\in\pd$ are  curvilinear polygons or even less regular, since they form a partition of $\hold$. 
In fact, an abstract structure theorem exists in the case of open sets, see \cite[Theorem 3.6, pp. 479-480]{MR2731611}, but when the shape derivative can be written as an integral a more explicit characterization is needed. 
In \cite[Theorem 1.3]{MR2350816}, a general structure theorem is proven, which shows that the shape derivative can be written as $L(\bta_{|\po} \cdot \bn)$
even when $\Om$ is only a set of finite perimeter, which is in particular valid for Lipschitz domains. 
However, the linear form $L$ is in general not a boundary integral if $\Om$ is only Lipschitz or piecewise $\C^\ke$.
For example, the shape derivative of the perimeter contains Dirac measures at the vertices of $\po$ when $\Om$ is a polygon; see \cite[Proposition 2.6]{MR2350816}.  

The structure \eqref{volume} can be seen as a {\it strong form} of the shape derivative, in the sense that it requires a strong regularity of the domain while it tolerates a low regularity of the vector field $\bta$.
In the multiphase context, it is natural to use weaker structures of the shape derivatives, which are valid for domains with low regularity but involve the derivatives of $\bta$ in return, which requires more regularity for $\bta$.
In the case where the functional is defined as a volume integral, its shape derivative can be written as a volume integral instead of a boundary integral, then we call it {\it distributed shape derivative}, see \eqref{eq:first_order_tensor}.
Also, it is sometimes possible to write shape derivatives as boundary integrals on Lipschitz domains as in \eqref{158}.
In this case, the structure is slightly weaker than \eqref{volume}, as the linear form depends on $\bta_{|\po }$ instead of $\bta_{|\po }\cdot \bn$. 
These weaker expressions, in particular the distributed shape derivative, are key ingredients of the LEM.
We discuss now some fundamental properties of weak  expressions of shape derivatives.
First of all, it is useful to write the distributed shape derivative using a tensor representation, as will be seen in Proposition \ref{tensor_relations}.
\begin{definition}[Tensor representation of distributed shape derivative]\label{def:tensor}
Let $\bfom\in\pd$ and assume $\mJ:\pd\mapsto \R$ has a shape derivative at $\bfom$.
The shape derivative of $\mJ$ admits a tensor representation of order $1$ if there exist a first-order tensor $\Sb_0\in L^1(\hold,\R^{d})$ and a second order tensor $S_1\in L^1(\hold,\R^{d\times d})$ such that for all $\bta \in \C^1_{\partial \hold}(\overline{\hold},\R^d)$, 
\begin{equation}
\label{eq:first_order_tensor}
d\mJ(\bfom)(\bta) = \int_\hold S_1 : D\bta +  \Sb_0\cdot \bta .
\end{equation}
\end{definition}
The following proposition extends the result \cite[Proposition 4.3]{MR3535238} to the multiphase case, also requiring weaker regularity assumptions.
\begin{proposition}\label{tensor_relations}
Assume $\bfom\in\pd$, $\bta \in \C^1_{\partial\hold}(\overline{\hold},\Rd)$, and $\mJ$ has a Eulerian shape derivative at $\bfom$ with the tensor representation \eqref{eq:first_order_tensor}.
If $S_1\in W^{1,1}(\Omega_k ,\R^{d\times d})$ for all  $k\in \K$, then
\begin{equation} \label{eq:equvilibrium_strong}
\begin{split}
\divv(S_1)  &= \Sb_0 \quad \text{ in } \Omega_k \text{ for all } k\in \K, \\
\end{split}
\end{equation}
and 
\begin{equation}\label{eq:first_order_tensor_2}
d\mJ(\bfom)(\bta) = \sum_{k\in\K}\int_{\Om_k} \divv(S_1^\transp\bta). 
\end{equation}
If in addition $\Om_k$ is Lipschitz for all $k\in \K$, then we have the boundary expression
\begin{equation}\label{158}
d\mJ(\bfom)(\bta) = \sum_{k\in\K}\int_{\partial \Omega_k} (S_{1,k} \bn_k) \cdot\bta . 
\end{equation}
where $S_{1,k}$ is the trace on $\partial \Omega_k$ of ${S_{1}}_{|{\Om_k}}$ and $\bn_k$ is the outward unit normal vector to $\Om_k$.
\end{proposition}
\begin{proof}
The proof is a straightforward adaptation to the multiphase context  of the proof of \cite[Proposition 1]{laurain_2ndorder_shape_2019}.
\end{proof}
\section{The lower envelope method}\label{sec:HJ}
In this section the notation $\blsf = (\lsf_0,\lsf_1, \dots ,\lsf_{\ck-1})$ stands for a vector of time-dependent  functions  $\lsf_k\in \C^\infty([0,\tz]\times\R^d,\R)$.
For simplicity we will sometimes use the notation $\lsf_k(t) : = \lsf_k(t,\cdot)$ and $\blsf(t) = (\lsf_0(t),\lsf_1(t), \dots ,\lsf_{\ck-1}(t))$.
The time-dependent phases $\Om_k(\blsf(t)), k\in\K$, are defined as in \eqref{def:omk}, and 
the interfaces $\mE_{\I}(\blsf(t)),\mM_{\I}(\blsf(t))$ and $\widehat\blsf_{\I}(t)$ as in Definition \ref{def2}.

\subsection{Interface tracking using the lower envelope approach}\label{sec:3.2}
For $t\in [0,t_0]$ and $k\in\K$, let $\bx(t)\in \partial\Om_k(\blsf(t))\cap\hold$ be a moving interface point.
Suppose that for all $t\in [0,t_0]$,  $|D\widehat\blsf_{\I}(t)| > 0$ on $\mM_{\I}(\blsf(t))$ for all $\I\in\mathds{I}_k^2$.
Then for each $t\in [0,t_0]$, we can apply Lemma \ref{lemma01} which yields that
$\bx(t)\in \mE_{\I}(\blsf(t))$ for some $\I\in\mathds{I}_k^r$ with $2\leq r\leq \ck$.
We also assume that we can choose $\I$ independent of $t$, and that the trajectory of $\bx(t)$ can be described by a flow of the type \eqref{Vxt} for some $\bta\in \C^{1}_{\partial \hold}(\overline{\hold},\R^d)$.
In view of \eqref{eq:01} and \eqref{res05} we have $\mE_\I(\blsf(t)) \subset \mM_{\I}(\blsf(t))$, consequently $\bx(t)$ satisfies the $\card-1$ equations
\begin{equation}\label{eq:2_0}
\lsf_k(t,\bx(t)) = \lsf_\ell(t,\bx(t)),\ \mbox{ for all } \ell\in\I\setminus\{k\}.
\end{equation}
Differentiating each of these relations with respect to $t$ yields for $t\in [0,t_0]$:
\begin{equation}
\label{eq:2}
\partial_t (\lsf_k - \lsf_\ell)(t,\bx(t)) + \bta(\bx(t))\cdot \nabla ( \lsf_k - \lsf_\ell)(t,\bx(t)) 
= 0,\ \mbox{ for all } \ell\in\I\setminus\{k\}.
\end{equation}
We extend equations \eqref{eq:2} to $\hold$, this yields  
\begin{equation}
\label{eq:2ext}
\partial_t (\lsf_k - \lsf_\ell)(t,x) + \bta(x)\cdot \nabla ( \lsf_k - \lsf_\ell)(t,x) 
= 0, \ \mbox{ for all } \ell\in\I\setminus\{k\},
\end{equation}
for $t\in [0,t_0]$ and $x\in\hold$.

Now, assume that there exists $\wblsf\in \C^\infty([0,\tz]\times\R^d,\R^\ck)$ solution of
\begin{align}
\label{eq:2b}
\partial_t \wlsf_k(t,x) + \bta(x)\cdot \nabla \wlsf_k(t,x) 
& = 0, \ \mbox{ for all } k\in\K, t\in [0,t_0] \text{ and } x\in\hold,\\
\label{eq:2bb}\wlsf_k(0,x) & =  \lsf_k(0,x),
\end{align}
where $\wlsf_k$ are the entries of $\wblsf$.
Then, for any $\I\in\mathds{I}_k^r$ with $2\leq r\leq \ck$, we have in view of \eqref{eq:2b} that $\wlsf_k - \wlsf_\ell$ satisfies \eqref{eq:2ext} for all $\ell\in\I\setminus\{k\}$.
Therefore, for small $t_0$ the phases $\Om_k(\wblsf(t))$ are a first-order approximation of $\Om_k(\blsf(t))$ for all $k\in\K$ and $t\in[0,t_0]$.
Thus, we will use  the transport equations  \eqref{eq:2b}-\eqref{eq:2bb} as the foundation of the LEM described in Section \ref{sec:LEM}.

In view of Lemma \ref{lemma01}, $\partial\Om_k(\blsf(t))$ is the union of all the sets  $\mE_{\I}(\blsf(t))$ with $\I\in\mathds{I}_k^2$. 
In practice, it is common that these sets are non-empty, see Examples \ref{example1}, \ref{example2}, \ref{example3D} and Figure \ref{fig1}.
If this is the case, then to describe the evolution of $\partial\Om_k(\blsf(t))$ we need to solve equations \eqref{eq:2ext} at least for all $\I\in\mathds{I}_k^2$.
Thus, in general we need to solve  \eqref{eq:2ext} for all $\ell\in\K\setminus\{k\}$, i.e. for $\ck-1$ equations.
Note that \eqref{eq:2b}-\eqref{eq:2bb} actually consists of $\ck$ equations, but in practice we can take $\psi_0 \equiv 0$ without loss of generality of the method, so in fact \eqref{eq:2b}-\eqref{eq:2bb} can be reduced to $\ck-1$ equations.
This shows that solving the transport equations \eqref{eq:2b}-\eqref{eq:2bb} for all $k\in\K$ is de facto a natural way of tracking the motion of interface points using the lower envelope representation of multiphases.
\subsection{Reducing the dimension of velocity fields}
An interesting question which naturally arises is to determine whether one needs to use the full vector field $\bta$ in \eqref{eq:2b}, or if the components of $\bta$ that are orthogonal to $\nabla \wlsf_k$ are superfluous.
For instance in the level set method \cite{MR965860}, which can be seen as a special case of the LEM for two phases (see Section \ref{sec:2phases}), one uses only the normal component of $\bta$ since the gradient of the level set function is orthogonal to the tangential component of $\bta$, in the case of smooth domains.
In the multiphase context however, the situation is more complicated due to the nonsmoothness of the sets $\Om_k(\blsf)$.

We now discuss this issue in more details.
Suppose that the assumptions of Lemma \ref{lemma1} and Lemma \ref{lemma003} are satisfied for all $t\in [0,\tz]$, then we have $\dim(\mE_{\I}(\blsf(t)))\leq \max\{d-(\card-1),0\}$ for all $\I\subset\K$ with $\card\ge 2$.
Assume for simplicity that  $\dim(\mE_{\I}(\blsf(t)))= \max\{d-(\card-1),0\}$.
Observe that  $\card-1$ is equal to the codimension of $\mE_{\I}(\blsf(t))$ with respect to the ambient space $\R^d$
and consider the decomposition 
$$\bta(\bx(t)) = \bta_\tau(\bx(t)) + \bta_\perp(\bx(t)),$$
with $\bta_\tau(\bx(t)) \in T_{\bx(t)} \mE_{\I}(\blsf(t))$ and $\bta_\perp(\bx(t)) \in (T_{\bx(t)} \mE_{\I}(\blsf(t)))^\perp$, where $T_{\bx(t)} \mE_{\I}(\blsf(t))$ is the tangent space of $\mE_{\I}(\blsf(t))$ at $\bx(t)\in \mE_{\I}(\blsf(t))$ of dimension  $\max\{d-(\card-1),0\}$, and  $(T_{\bx(t)} \mE_{\I}(\blsf(t)))^\perp$ is its orthogonal complement in $\R^d$, of dimension $\min\{\card-1,d\}$.
Then, for $k\in\I$, in view of \eqref{eq:2_0} one observes that  $\nabla ( \lsf_k - \lsf_\ell)(t,\bx(t))\in (T_{\bx(t)} \mE_{\I}(\blsf(t)))^\perp$ for all $\ell\in\I\setminus\{k\}$, consequently equations \eqref{eq:2} become
\begin{equation}
\label{eq:2z}
\partial_t (\lsf_k - \lsf_\ell)(t,\bx(t)) + \bta_\perp(\bx(t))\cdot \nabla ( \lsf_k - \lsf_\ell)(t,\bx(t)) 
= 0, \ \mbox{ for all } \ell\in\I\setminus\{k\}.
\end{equation}
In the particular case $\I = \K = \{0,1\}$ and $\lsf_0\equiv 0$, which corresponds to the LSM, we have $\bta_\perp = (\bta\cdot \bn)\bn$ where $\bn$ is the outward unit normal vector to $\Om_0(\blsf)$. 
This corresponds to the standard simplification made in the LSM which yields the level set equation; see \cite{MR1700751}.

We may also relate this observation to the structure theorem  \cite[Corollary 5.6]{MR3593535}, where it is proved that the shape derivative of functionals taking smooth manifolds of dimension $d_0$ in $\R^d$ as argument only depends on the component $\bta_\perp$ of dimension $d-d_0$. 
Taking $d_0 = d-(\card-1)$ we arrive at the same conclusion, i.e. that it is sufficient to use $\bta_\perp$ to track the motion of $\mE_{\I}(\blsf(t))$, as in \eqref{eq:2z}.

We observe, however, that the dimension of  $(T_{\bx(t)} \mE_{\I}(\blsf(t)))^\perp$ depends on $\I$, and that in view of Lemma \ref{lemma01}, $\partial\Om_k(\blsf(t))$ is typically the union of sets $\mE_{\I}(\blsf(t))$ whose dimensions take all integer values between $0$ and $d-1$.
In particular,  when $\card\geq d+1$, then Lemma \ref{lemma003} indicates that $\mE_{\I}(\blsf(t))$ is a set of isolated points and $\dim(T_{\bx(t)} \mE_{\I}(\blsf(t)))^\perp=\min\{\card-1,d\}= d$, so that $\bta_\perp = \bta $ has dimension $d$.

In this case, one is constrained to use the full vector $\bta$ to describe the evolution of $\partial\Om_k(\blsf(t))$, at least locally around the sets $\mE_{\I}(\blsf(t))$ with zero dimension.
This shows that the lowest-dimensional subsets $\mE_{\I}(\blsf(t))$ of  $\partial\Om_k(\blsf(t))$ dictate the dimension of the vector field $\bta_\perp$ that should be used to track the motion of $\partial\Om_k(\blsf(t))$.
From  the point of view of numerical implementation, this is in accordance with the use of weak forms of shape derivatives such as \eqref{eq:first_order_tensor} or \eqref{158}, where the full vector $\bta$ is naturally available rather than  $\bta_\perp$.
This is a generalization of the idea used in \cite{MR3535238}, where the full vector $\bta$ was used in a distributed shape derivative-based level set method instead of the normal component $\bta\cdot\bn$ used in the LSM.

\subsection{Narrow band approach}
In the LSM, the level set equations can be solved in a small neighbourhood of the interface to decrease the computational cost, this is the so-called {\it narrow band} approach. 
In the case of the lower envelope method, one could also use the same idea and solve equations \eqref{eq:2b}-\eqref{eq:2bb} in a small neighbourhood of the union of all interfaces $\cup_{\I\in\mathds{I}^r_2,r\ge 2}\mE_{\I}(\blsf(t))$.
\subsection{Description of the lower envelope method}\label{sec:LEM}
We now have the theoretical foundation to describe the LEM.
Given $\blsf_0 \in \C^\infty(\R^d,\R^\ck)$ an initial vector-valued function, a vector field $\bta\in \C^1_{\partial\hold}(\overline{\hold},\Rd)$ and the associated flow
$\Tt^{\bta}:\overline{\hold}\rightarrow \Rd$, find $\blsf \in \C^\infty([0,\tz]\times\R^d,\R^\ck)$ solution of the transport equations  
\begin{align}
\label{eq_env0} \partial_t \lsf_{k}(t,x) + \bta(x)\cdot \nabla  \lsf_{k}(t,x) &= 0,\qquad  \mbox{ for } t\in [0,t_0] \mbox{ and }x\in\hold,\\
\label{eq_env1} \lsf_{k}(0,x) &= \lsf_{k,0}(x),
\end{align}
for all  $k\in\K$, where $\lsf_k,\lsf_{k,0}$ are the entries of $\blsf,\blsf_0$, respectively.
The moving vector domain is defined as $\bfom_t: = (\Om_0(\blsf(t)),\dots,\Om_{\ck-1}(\blsf(t)))$, where $\Om_k(\blsf(t))$ is defined as in \eqref{def:omk}.
Note that we can write \eqref{eq_env0}-\eqref{eq_env1} in an equivalent way in vectorial form as
\begin{align}
\label{eq_env0b} \partial_t \blsf(t,x) + D\blsf(t,x) \bta(x)  &= 0,\qquad \mbox{ for } t\in [0,t_0] \mbox{ and }x\in\hold,\\
\label{eq_env1b} \blsf(0,x) &= \blsf_{0}(x).
\end{align}
By analogy with the LSM, we call \eqref{eq_env0b}-\eqref{eq_env1b} the {\it lower envelope equation}.
If we assume that for all $t\in [0,t_0]$ we have  $|D\widehat\blsf_{\I}(t)| > 0$ on $\mM_{\I}(\blsf(t))$ for all $\I\in\mathds{I}_k^2$, then this guarantees that $\bfom_t\in\pd$  for all $t\in [0,t_0]$ in view of Theorem \ref{thm1}.

In a practical implementation, we may choose $\lsf_{0,0}\equiv 0$ which yields $\lsf_{0}(t)\equiv 0$ for all $t\in [0,t_0]$. 
This does not reduce the generality of the method and is less expensive from a computational point of view.
For shape optimization problems, $\bta$ is usually chosen as a descent direction for the multiphase cost functional $\mJ : \pd\rightarrow \R$, which can be obtained by solving an elliptic PDE using a weak form of the shape derivative on the right-hand side; see Section \ref{sec:algo} for more details on the procedure.

\subsection{The particular case of two phases}\label{sec:2phases}
In the case $\K =\{0,1\}$, $\lsf_1 \in \C^\infty(\R^d,\R)$ and $\lsf_0 \equiv 0$, we show that the LEM coincides with the LSM \cite{MR965860}. 
First of all, assuming $0$ is a regular value of $\lsf_1$, Definition \ref{def1} yields 
\begin{equation*}
\Om_1(\blsf) := \{x\in\hold\ |\ \lsf_1(x) <0  \},\quad \Om_0(\blsf) :=\{x\in\hold\ |\ \lsf_1(x) >0  \},
\end{equation*}
which corresponds to the definition of the domains in the level set method.

Then, the lower envelope equation \eqref{eq_env0b} reduces to the following transport equation
\begin{equation}\label{104}
\partial_t \lsf_1(t,x) + \bta(x)\cdot \nabla \lsf_1(t,x) 
= 0.
\end{equation}
Assuming $0$ is a regular value of $\lsf_1$, then $\Om_1(\blsf)$ is smooth and \eqref{104} reduces to the usual level set equation 
\begin{equation}
\partial_t \lsf_1(t,x) + \bta(x)\cdot\bn(x) |\nabla \lsf_1(t,x)| = 0.
\end{equation}
This shows that the LSM is a particular case of the LEM  using two phases.

\section{Geometric properties of the LEM}\label{sec:geo_lem}

\subsection{Properties of triple points in two dimensions}\label{sec:prop}

In this section we assume $d=2$, $\K=\{0,1,2\}$, $\blsf = (\lsf_0,\lsf_1,\lsf_2)\in \C^\infty(\R^2,\R^{3})$, and $\lsf_0 \equiv 0$. 
In this situation there are three interfaces $\mE_{\{0,1\}}(\blsf)$, $\mE_{\{1,2\}}(\blsf)$ and $\mE_{\{0,2\}}(\blsf)$ of dimension one, and assuming $\mE_\K(\blsf)$ is not empty, $\mE_\K(\blsf)$ is a set of triple points  according to Lemma \ref{lemma003} and Definition \ref{def3}; see Figure \ref{fig2} for an illustration.
Let $\hat{x}\in \mE_\K(\blsf)$ be a triple point.
Denote $\mathds{D}_\I$ the half-tangent to $\mE_\I(\blsf)$ at $\hat{x}$ for $\I =\{0,1\}$, $\I =\{1,2\}$ or $\I =\{0,2\}$.
Denote $\vartheta\in [0,2\pi]$ the angle in local polar coordinates with origin $\hat{x}$ and such that $\vartheta=0$ corresponds to  $\mathds{D}_{\{0,2\}}$.
Without loss of generality, we may assume that $\vartheta_0\leq \vartheta_1$, where $\vartheta_0$ is the angle  between $\mathds{D}_{\{0,2\}}$ and $\mathds{D}_{\{0,1\}}$ and $\vartheta_1$ is the angle between $\mathds{D}_{\{0,2\}}$ and $\mathds{D}_{\{1,2\}}$.
Indeed, if $\vartheta_0> \vartheta_1$ we can just exchange the indices of $\lsf_0$ and $\lsf_2$, rename the phases accordingly, and we will get $\vartheta_0\leq \vartheta_1$.
Introduce also the relative angles $\beta_0 = \vartheta_0\ge 0$, $\beta_1 = \vartheta_1 - \vartheta_0\ge  0$ and $\beta_2 = 2\pi - \vartheta_1\ge  0$.
Clearly, we have $\beta_0+\beta_1+\beta_2 = 2\pi$; see Figure \ref{fig2}.
\begin{thm}\label{thm:3}
Let $\hat{x}\in \mE_\K(\blsf)$ and assume $D\widehat\blsf_{\K}(\hat{x})$ is invertible, then $\max_{k\in\K} \beta_k <\pi$ and $\min_{k\in\K} \beta_k >0$. 
\end{thm}
\begin{proof}
First we assume that $\beta_0>\pi$ and show that this leads to a contradiction.
Without loss of generality we may assume that $\mathds{D}_{\{0,2\}}$ coincides with the right semiaxis $Ox$.
Since $\beta_0>\pi$, $\mathds{D}_{\{0,1\}}$ and $\mathds{D}_{\{1,2\}}$ must be both located in the open lower half-plane.

Denote $\mathds{H}_r := \{(x,y)\ |\ x>0\}$ the right open half-plane and $\mathds{H}_l := \{(x,y)\ |\  x<0\}$ the left open half-plane.
For $x\in\mE_{\{0,2\}}(\blsf)$ we have $\nabla_\Gamma\lsf_2(x) = 0$ since $\lsf_2 = \lsf_0$ on $\mE_{\{0,2\}}(\blsf)$ and $\lsf_0\equiv 0$, where $\nabla_\Gamma$ denotes the tangential gradient  on $\mE_{\{0,2\}}(\blsf)$.
We also have $\nabla\lsf_2(x)\cdot\bn_2(x) > 0$ for all $x\in\mE_{\{0,2\}}(\blsf)$, where $\bn_2$ is the unit outward normal vector to $\Om_2(\blsf)$, since $\lsf_2\leq\lsf_0$ in $\Om_2(\blsf)$ and $\lsf_2\geq\lsf_0$ in $\Om_0(\blsf)$.  
As $\mathds{D}_{\{0,2\}}$ coincides with the right semiaxis $Ox$, we get $\nabla\lsf_2(\hat{x}) = (0,\lambda)$ with $\lambda>0$.
In a similar way we have that $\nabla\lsf_1(\hat{x})$ is orthogonal to $\mathds{D}_{\{0,1\}}$ and $-\nabla\lsf_1(\hat{x})\in\mathds{H}_r$, using the fact  that $\nabla\lsf_1(\hat{x}) \neq 0$, thanks to the assumption that $D\widehat\blsf_{\K}(\hat{x})$ is invertible.
Thus, we have shown that $\nabla (\lsf_2 - \lsf_1)(\hat{x}) \in\mathds{H}_r$.

In a similar way we have that $\nabla(\lsf_2 -\lsf_1)(\hat{x})$ is orthogonal to $\mathds{D}_{\{1,2\}}$.
The fact that $\lsf_2-\lsf_1 \geq 0$ in $\Om_1$ and  $\lsf_2-\lsf_1 \leq 0$ in $\Om_2$ shows that  $\nabla(\lsf_2 -\lsf_1)(\hat{x})$ is pointing outward of $\Om_2$, therefore it must be in $\mathds{H}_l$.
Thus, we have obtained $\nabla (\lsf_2 - \lsf_1)(\hat{x}) \in\mathds{H}_r\cap \mathds{H}_l$ which is a contradiction since $\mathds{H}_r \cap \mathds{H}_l =\emptyset$, and this implies that $\beta_0\leq \pi$. 
In a similar way, one also proves $\beta_k\leq \pi$ for $k=1,2$.

Now assume that $\beta_0 = \pi$, then $\nabla\lsf_2(\hat{x})$ and $\nabla\lsf_1(\hat{x})$ are linearly dependent which implies $\det D\widehat\blsf_{\K}(\hat{x})=0$, and this contradicts the assumption that  $D\widehat\blsf_{\K}(\hat{x})$ be invertible.
Hence, we must have $\beta_0 < \pi$ and also $\beta_1 < \pi$, $\beta_2 < \pi$ in a similar way.
Then, the fact that  $\min_{k\in\K} \beta_k >0$ is a straightforward consequence of  $\beta_0+\beta_1+\beta_2 = 2\pi$.
\end{proof}
\begin{figure}
\begin{center}
\begin{tikzpicture}[scale=8.0]
\draw[thick]  (0,0) --(1,0) --(1,1)--(0,1)--(0,0);
\draw[dashed,thick]  (0,0)--(0.5,0.5);
\draw[dashed,thick]  (0,1)--(0.5,0.5);
\draw[dashed,thick]  (0.5,0.5)--(1,0.5);
\node[text=black] at (0.9,0.9){$\Om_0(\lsf)$};
\node[text=black] at (0.9,0.1){$\Om_2(\lsf)$};
\node[text=black] at (0.1,0.55){$\Om_1(\lsf)$};
\node[text=black] at (0.15,0.35){$\mE_{\{1,2\}}(\blsf)$};
\node[text=black] at (0.4,0.8){$\mE_{\{0,1\}}(\blsf)$};
\node[text=black] at (0.85,0.6){$\mE_{\{0,2\}}(\blsf)$};
\node[text=black] at (0.85,0.45){$\mathds{D}_{\{0,2\}}$};
\node[text=black] at (0.25,0.1){$\mathds{D}_{\{1,2\}}$};
\node[text=black] at (0.15,0.75){$\mathds{D}_{\{0,1\}}$};
\draw [black,->,thick,domain=10:125] plot ({0.5+0.1*cos(\x)}, {0.5+0.1*sin(\x)});
\draw [black,->,thick,domain=-125:-10] plot ({0.5+0.1*cos(\x)}, {0.5+0.1*sin(\x)});
\draw [black,->,thick,domain=145:215] plot ({0.5+0.1*cos(\x)}, {0.5+0.1*sin(\x)});
\node[text=black] at (0.55,0.35){$\beta_2$};
\node[text=black] at (0.55,0.65){$\beta_0$};
\node[text=black] at (0.35,0.5){$\beta_1$};
\node[text=black] at (0.52,0.54){$\hat{x}$};
\draw[thick,smooth,samples=100,domain=0.5:1] plot(\x,{0.485+0.06*\x*\x});
\draw[thick,smooth,samples=100,domain=0:0.5] plot(\x,{0.25+\x*\x});
\draw[thick,smooth,samples=100,domain=0.1:0.5] plot(\x,{15*\x*\x/24 - 39*\x/24 + 27.75/24});
\end{tikzpicture}
\caption{The sets $\hold = (0,1)^2$ and $\mE_{\I}(\blsf)$, half-tangents $\mathds{D}_{\I}$  for $\I = \{0,1\},\{0,2\},\{1,2\}$  and angles $\beta_0,\beta_1,\beta_2$. }\label{fig2}
\end{center}
\end{figure}
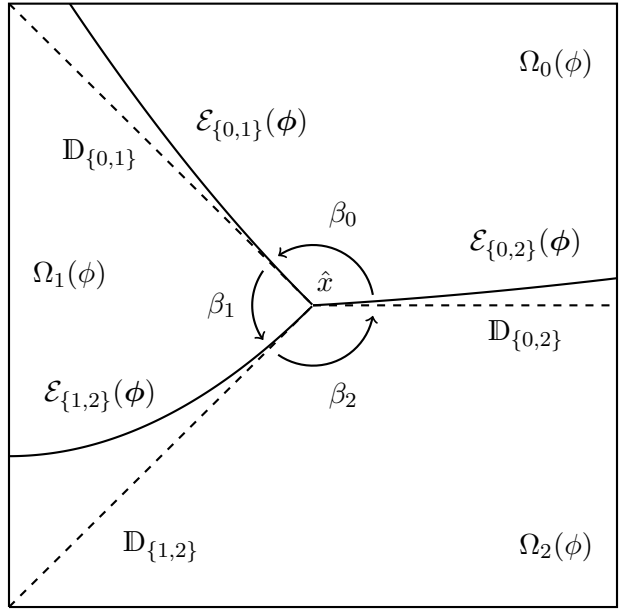

We can also compute the angles at the triple point $\hat{x}$ in the following way.
\begin{proposition}
Let $\hat{x}\in \mE_\K(\blsf)$ and assume $D\widehat\blsf_{\K}(\hat{x})$ is invertible, then 
\begin{equation}\label{betak}
\beta_k = \arccos \frac{\nabla (\lsf_{[k+1]_3} - \lsf_k)\cdot \nabla (\lsf_k - \lsf_{[k+2]_3})}{|\nabla (\lsf_{[k+1]_3} - \lsf_k)|\cdot | \nabla (\lsf_k - \lsf_{[k+2]_3})|}\ \text{ for all } k\in\K,
\end{equation}
where $[k+1]_3$ means $k+1$ modulo $3$.
\end{proposition}
\begin{proof}
The vector $\nabla(\lsf_2 - \lsf_1)$ is orthogonal to $\mathds{D}_{\{1,2\}}$ and points outward of $\Om_2(\blsf)$, while
the vector $\nabla(\lsf_1 - \lsf_0)$ is orthogonal to $\mathds{D}_{\{0,1\}}$ and points outward of $\Om_1(\blsf)$.
Hence,  $\beta_1$ is also the angle between $\nabla(\lsf_2 - \lsf_1)$ and $\nabla(\lsf_1 - \lsf_0)$, and since $0<\beta_1<\pi$ according to Theorem \ref{thm:3}, this yields \eqref{betak} for $k=1$. 
The other cases are obtained in the same way. 
\end{proof}

\subsection{Evolution of $(d+1)$-tuple points}\label{sec:evol_triple}

We have formally shown in Section \ref{sec:3.2} that the  lower envelope equation \eqref{eq_env0b}-\eqref{eq_env1b} represents a first-order approximation of the motion of  interfaces  $\partial\Om_k(\blsf(t))$ for all $k\in\K$.
Nevertheless, we would like to verify that the  lower envelope equation \eqref{eq_env0b}-\eqref{eq_env1b} indeed leads to the motion of $(d+1)$-tuple points with the expected velocity $\bta$ in a neighbourhood of $t=0$.
The main tool to achieve this result is the implicit function theorem.

Suppose  $\I=\K$,  $\ck = d+1$, and $D\widehat\blsf_{\I}(0,x)$ is invertible for all $x\in \mM_\I(\blsf(0))$, then  $\mE_\I(\blsf(0))=\mM_\I(\blsf(0))$ is a set of isolated points  in view of Lemma \ref{lemma003}.
Without loss of generality, we can assume that $\mM_\I(\blsf(0)) = \{\hat{x}\}$ is exactly one point. 
Then we would like to determine the behaviour of $\mM_\I(\blsf(t))$ for small $t$.
In view of \eqref{eq:malpha0} we have
$$\mM_{\I}(\blsf(0)) = \{x\in\overline{\hold}\ |\ \widehat\blsf_\I(0,x) =0\} = \{\hat{x}\}.$$
Using that $D\widehat\blsf_\I(0,x)$ is invertible for all $x\in \mM_\I(\blsf(0))$, and possibly reducing $t_0$,  the implicit function theorem yields the existence of a unique function $\bx^\dagger: [0,t_0]\to \R^d$ such that $\bx^\dagger(0) =\hat{x}$ and 
$\widehat\blsf_\I(t,\bx^\dagger(t)) =0$ for all $t\in [0,t_0]$.
Thus, we get
\begin{equation}\label{eq:34}
\bx^\dagger(t) \in \mM_{\I}(\blsf(t)) = \{x\in\overline{\hold}\ |\ \widehat\blsf_\I(t,x) =0\}.
\end{equation}
Reducing $t_0$ if necessary, we also have that $D\widehat\blsf_\I(t,x)$ is invertible for all $x\in \mM_\I(\blsf(t))$ and all $t\in [0,t_0]$.
Thus, applying Lemma \ref{lemma003} using $\I=\K$ and  $\ck = d+1$, we have $\mE_{\I}(\blsf(t))=\mM_{\I}(\blsf(t))$ and \eqref{eq:34} yields that $\bx^\dagger(t)$ is a $(d+1)$-tuple point for all $t\in [0,t_0]$.
The implicit function theorem also yields for the derivative
\begin{equation}\label{eq:77}
\partial_t \widehat\blsf_\I(t,\bx^\dagger(t)) + D\widehat\blsf_\I(t,\bx^\dagger(t)) \partial_t\bx^\dagger(t)  = 0,\quad \mbox{ for } t\in [0,t_0]. 
\end{equation}
Taking the difference between the equations for $\lsf_k$ and $\lsf_\ell$ at  $x = \bx^\dagger(t)$ in \eqref{eq_env0}, 
and subtracting the result to  \eqref{eq:77} yields
\begin{equation*}
D\widehat\blsf_\I(t,\bx^\dagger(t)) (\partial_t\bx^\dagger(t) - \bta (\bx^\dagger(t)) ) = 0,\quad \mbox{ for } t\in [0,t_0]. 
\end{equation*}
Using that $D\widehat\blsf_\I(t,\bx^\dagger(t))$ is invertible for $t\in [0,t_0]$ we get 
$$\partial_t \bx^\dagger(t) = \bta (\bx^\dagger(t)) ,\quad \mbox{ for } t\in [0,t_0].$$
This shows that the lower envelope equation \eqref{eq_env0b}-\eqref{eq_env1b} leads to the  evolution of the $(d+1)$-tuple point $\bx^\dagger(t)$ with the expected velocity  $ \bta (\bx^\dagger(t))$ for sufficiently small time $t\in [0,t_0]$. 

Now we consider the case $\ck > d+1$.
Suppose that $\{\hat{x}\}\in\mM_\K(\blsf(0))$, then $\hat{x}$ is at the junction of $\ck > d+1$ phases. 
Then, there exists at least two different subsets $\I^0\subset\K$ and  $\I^1\subset\K$ with $|\I^0| = |\I^1| =d+1$ such that $\hat{x} \in\mM_{\I^0}(\blsf(0))$ and $\hat{x} \in\mM_{\I^1}(\blsf(0))$.
We can then proceed with the same reasoning as above, except that we only have the inclusions $\mE_{\I^0}(\blsf(0))\subset\mM_{\I^0}(\blsf(0))$ and $\mE_{\I^1}(\blsf(0))\subset\mM_{\I^1}(\blsf(0))$ instead of equalities.
On one hand, this means that  $\mE_{\I^0}(\blsf(t))$ and $\mE_{\I^1}(\blsf(t))$ could be empty for $t>0$.
On the other hand, even if we assume that both sets are non-empty for all $t\in [0,\tz]$, possibly reducing $\tz$, we obtain two functions $\bx^\dagger_0: [0,t_0]\to \R^d$ and $\bx^\dagger_1: [0,t_0]\to \R^d$ that both satisfy $\partial_t \bx^\dagger_i(t) = \bta (\bx^\dagger_i(t))$, $i=0,1,$ but are not necessarily equal.

We conclude that $(d+1)$-tuple points are stable in the case $\I=\K$ and $\ck = d+1$, while multiple junctions are unstable for  $\ck > d+1$ in the sense that they can split and move in different directions for $t>0$.

\section{Application to an inverse conductivity problem}\label{sec:EIT}
\subsection{Problem formulation}
We consider the inverse problem of determining a matrix-valued conductivity $\sigma$ of a body $\hold\subset\R^d$ satisfying
the elliptic equations 
\begin{equation}\label{eq:eit}
\divv(\sigma \nabla u_i) = f\mbox{ in }\hold,
\end{equation}
where $u_i$, $i=1,\dots,m$ are the potentials associated with applied boundary current fluxes $g_i=\sigma \nabla u_i\cdot\bn|_{\Sigma}$, and measurements of boundary voltages $\meas_i=u_i|_\Sigma$ are available on an open subset $\Sigma$ of $\partial\Om$. 

When $f\equiv0$, this problem is known as the continuum model in electrical impedance tomography (EIT), also known as the Calder\'on problem; we refer to the reviews \cite{Bera_2018,MR1955896} and the references therein.
There exists a vast literature on EIT in the isotropic case, which corresponds to $\sigma = \gamma I_d$, where $I_d$ is the identity matrix and $\gamma$ is a scalar-valued function, but there are much less known results in the anisotropic case; however one should mention \cite{MR3810174} for uniqueness results in the case of a layered anisotropic medium.
Here, we compute the shape derivative in the multiphase anisotropic case, and for the numerics we focus on the isotropic case.

Introduce 
$$\sigma=\sigma_{\bfom} := \sum_{k\in\K} \sigma_k\chi_{\Om_k}\ \text{ and }\  f=f_{\bfom} := \sum_{k\in\K} f_k\chi_{\Om_k},$$ 
where  $\chi_{\Omega_k}$ denotes the characteristic function of $\Omega_k$, $\sigma_k$ are matrix-valued functions and $\bfom\in\pd$; see Definition \ref{def4}.
\begin{assump}\label{assump1}
We make the following assumptions on the material parameters for all $k\in\K$:
\begin{itemize}
\item $\bfom\in\pd$,
\item $\sigma_k:\overline{\hold}\to\R^{d\times d}$ is assumed to be $\C^1(\overline{\hold})$ and uniformly positive definite, i.e., there exists $\underline{\sigma}$ (independent of $k$) such that
$ \xi^\transp\sigma_k(x)\xi \geq \underline{\sigma} |\xi|^2$  for a.e.  $x\in\overline{\hold} \text{ and all }\xi\in\R^d$,
\item $\sigma_k \not\equiv \sigma_\ell$ for all $k\neq \ell$,
\item $f_k\in H^1(\hold)$. 
\end{itemize}
\end{assump}
In order to obtain a numerical approximation of the solution of the EIT problem, we consider a Kohn-Vogelius approach with mixed boundary conditions as in \cite{MR3535238}.
For $i=1,\dots,m$, introduce $\ua_i\in H^{1}_{a,h}(\hold)$ and $\ub_i\in H^{1}_{b,h}(\hold)$ solutions of
\begin{align}\label{E:var_form1}
\int_{\hold} \sigma_{\bfom} \nabla \ua_i \cdot \nabla \test &= \int_\hold f_{\bfom} \test + \int_{\Gamma_b} g_i\test  \quad \mbox{ for all }\test\in H^{1}_{a,0}(\hold), \\
\label{E:var_form2} \int_{\hold} \sigma_{\bfom} \nabla \ub_i \cdot \nabla \test &= \int_\hold f_{\bfom} \test + \int_{\Gamma_a} g_i\test  \quad \mbox{ for all }\test\in H^{1}_{b,0}(\hold), 
\end{align}
with $\overline{\Gamma_a}\cup\overline{\Gamma_b} = \partial\hold$, $\Gamma_a\neq\emptyset$, $\Gamma_b\neq\emptyset$,  $g_i\in H^{-1/2}(\Gamma)$, $h_i\in H^{1/2}(\Gamma)$ and
\begin{align*}
H^{1}_{a,h}(\hold) & := \{\test\in H^1(\hold)\ |\ \test = h_i\text{ on } \Gamma_a  \},\\
H^{1}_{b,h}(\hold) & := \{\test\in H^1(\hold)\ |\ \test = h_i\text{ on } \Gamma_b  \}.
\end{align*}
The inverse problem then consists in finding $\sigma_{\bfom}$ such that $\ua_i=\ub_i$ for all $i=1,\dots,m$.
However,  the measurements $h_i$ are corrupted by noise in practice, therefore
we cannot expect that  $\ua_i = \ub_i$ be exactly achievable, but rather that $|\ua_i - \ub_i|$ should be minimized.
Thus, we shall consider the following multiphase cost functional:
\begin{align}
\label{eit3.2} \mJ(\bfom) & := \frac{1}{2}\sum_{i=1}^m\int_{\hold} (\ua_i -  \ub_i)^2.
\end{align}

\subsection{Shape derivative of the cost functional}
In this section we take $m=1$  and we write $\ua,\ub,g,\meas$ instead of $\ua_1,\ub_1,g_1,\meas_1$ to simplify the notation.
The expression of the shape derivative in the  case $m>1$ can be obtained straightforwardly by summing over $i=1,\dots,m$.

Before stating the main result of this section, a short discussion about third-order tensors is useful. 
During the calculation of the shape derivative of $\mJ(\bfom)$ appears the term 
$$\widetilde\sigma(t): =  \sum_{k\in\K} \chi_{\Om_k}\sigma_k\circ\Tt$$
whose derivative is given by
$$\widetilde\sigma'(0)=  \sum_{k\in\K} \chi_{\Om_k} D\sigma_k\bta .$$
Here, $D\sigma_k:\hold\to \R^{d\times d\times d}$ is a third-order tensor with entries $(\partial_\ell (\sigma_k)_{ij})_{ij\ell}$.
Let  $\mathds{A}\in\R^{d\times d\times d}$ and $\mathds{B}\in\R^{d\times d\times d}$ be two third-order tensors satisfying
$$ \mathds{A} y z\cdot x = \mathds{B} z x\cdot y \quad\text{ for all } x, y,z\in\R^d.$$
Then we call $\mathds{B}$ the transpose of $\mathds{A}$ and we write $\mathds{B} = \mathds{A}^\transp$.  
It can be shown that the transpose of $\mathds{A}$ always exists and is unique; see
\cite[Proposition 3.1]{qi2017transposes}.

For instance, the term  $D\sigma_k^\transp \nabla \ua \nabla \pa \cdot\bta$ appearing in  $\Sb_0(\bfom)\cdot \bta$ in \eqref{T:tensor_shape_deriv}
can be computed as follows: $D\sigma_k^\transp \nabla \ua \nabla \pa \cdot\bta = D\sigma_k \bta\nabla \ua \cdot \nabla \pa  = \sum_{i,j,\ell = 1}^d \partial_\ell (\sigma_k)_{ij} \bta_\ell\partial_j\ua\partial_i\pa$, which means that  $D\sigma_k^\transp \nabla \ua \nabla \pa$ is a vector with entries $(\sum_{i,j = 1}^d \partial_\ell (\sigma_k)_{ij} \partial_j\ua\partial_i\pa)_\ell$.

\begin{thm}[distributed shape derivative]\label{T:shape}
Let Assumption \ref{assump1} be satisfied, then the shape derivative of $\mJ$ at $\bfom$ in direction $\bta\in\C^1_{\partial\hold}(\overline{\hold},\R^d)$
is given by 
\begin{align}
\label{T:tensor_shape_deriv}
d\mJ(\bfom)(\bta)  = \int_{\hold} S_1(\bfom): D \bta  +  \Sb_0(\bfom)\cdot  \bta  \;dx,
\end{align}
where $S_1(\bfom)\in L^1(\hold,\R^{d\times d})$ and $\Sb_0(\bfom)\in L^1(\hold,\R^d)$ are defined by 
\begin{align}
\label{S1_1} S_1(\bfom) & = \left[\frac{1}{2}(\ua-\ub)^2 
-f_{\bfom}(\pa +\pb)
+  \sigma_{\bfom}\nabla \ua\cdot\nabla \pa + \sigma_{\bfom}\nabla \ub\cdot\nabla \pb 
\right]  I_d
\\
\notag&\quad 
- \nabla \pa\otimes \sigma_{\bfom}\nabla \ua 
- \nabla \ua\otimes \sigma_{\bfom}^\transp\nabla \pa 
- \nabla \pb\otimes \sigma_{\bfom}\nabla \ub 
- \nabla \ub\otimes \sigma_{\bfom}^\transp\nabla \pb,\\
\label{S0} \Sb_0(\bfom) & =  \sum_{k\in\K} \chi_{\Om_k} [D\sigma_k^\transp \nabla \ua \nabla \pa  + D\sigma_k^\transp \nabla \ub \nabla \pb 
-(p+q){\nabla} f_k ],  
\end{align}
where  $D\sigma_k^\transp $ denotes the transpose of the third-order tensor  $D\sigma_k:\hold\to \R^{d\times d\times d}$.

The adjoints $\pa\in H^1_{a,0}(\hold)$ and $\pb\in H^1_{b,0}(\hold)$ are solutions of
\begin{align}
\label{57a}\int_\hold \sigma_{\bfom}^\transp \nabla \pa \cdot \nabla \test &= -\int_{\hold} (\ua - \ub) \test\quad \mbox{ for all } \test\in H^1_{a,0}(\hold),\\
\label{57b}\int_\hold \sigma_{\bfom}^\transp \nabla \pb \cdot \nabla \test &= \int_{\hold} (\ua - \ub) \test\quad \mbox{ for all } \test\in H^1_{b,0}(\hold).
\end{align}
\end{thm}
\begin{proof}
We use the averaged adjoint method \cite{MR3374631} to compute the shape derivative of $\mJ(\bfom)$.
The  existence proof for the shape derivative of $\mJ(\bfom)$ goes in a similar way as in \cite{MR3535238}, where the isotropic case for two phases was treated.
Therefore, we only show the calculation here, and we refer to \cite{MR3535238} for the verification of the assumptions of  the averaged adjoint method.

First of all, in order to avoid working with  $H^{1}_{a,h}(\hold)$ and $H^{1}_{b,h}(\hold)$, we introduce  alternative variational formulations equivalent to \eqref{E:var_form1}-\eqref{E:var_form2}: find $\ua\in H^{1}(\hold)$ and $\ub\in H^{1}(\hold)$ solutions to
\begin{align}\label{E:var_form3}
\int_{\hold} \sigma_{\bfom} \nabla \ua \cdot \nabla \test &= \int_\hold f_{\bfom} \test 
+ \int_{\Gamma_b} g\test  
+  \int_{\Gamma_a}(\sigma_{\bfom}^\transp \nabla\test)\cdot\bn  (\ua - \meas) \quad \mbox{ for all }\test\in H^{1}_{a,0}(\hold), \\
\label{E:var_form4} \int_{\hold} \sigma_{\bfom} \nabla \ub \cdot \nabla \test 
&= \int_\hold f_{\bfom} \test + \int_{\Gamma_a} g\test  
+  \int_{\Gamma_b}(\sigma_{\bfom}^\transp \nabla\test)\cdot\bn  (\ub - \meas) \quad \mbox{ for all }\test\in H^{1}_{b,0}(\hold).
\end{align}
Note that the integrals on $\Gamma_a$ and $\Gamma_b$ in \eqref{E:var_form3}-\eqref{E:var_form4}  should be understood as dual products since $g$ and $(\sigma_{\bfom}^\transp \nabla\test)\cdot\bn$ belong to $H^{-1/2}(\Gamma)$. 
Compared to \eqref{E:var_form1}-\eqref{E:var_form2}, the additional terms in \eqref{E:var_form3}-\eqref{E:var_form4}  yield the non-homogeneous Dirichlet conditions $\ua=\meas$ on $\Gamma_a$ and $\ub=\meas$ on $\Gamma_b$.  

We define the Lagrangian $\mathcal{L}:\pd\times H^{1}(\hold)\times H^{1}(\hold)\times H^{1}_{a,0}(\hold) \times H^{1}_{b,0}(\hold)$ as
\begin{align*}
\mathcal{L}(\bfom,(\xa,\xb),(\ya,\yb)) &:= \frac{1}{2}\int_{\hold} (\xa -  \xb)^2 
+ \int_{\hold} \sigma_{\bfom} \nabla \xa \cdot \nabla \ya - \int_\hold f_{\bfom} \ya 
- \int_{\Gamma_b} g\ya 
- \int_{\Gamma_a}(\sigma_{\bfom}^\transp \nabla\ya)\cdot\bn  (\xa - \meas)\\
&\quad + \int_{\hold} \sigma_{\bfom} \nabla \xb \cdot \nabla \yb - \int_\hold f_{\bfom} \yb 
- \int_{\Gamma_a} g\yb
- \int_{\Gamma_b}(\sigma_{\bfom}^\transp \nabla\yb)\cdot\bn  (\xb - \meas). 
\end{align*}
Then,  adjoints $\pa\in H^1_{a,0}(\hold)$ and $\pb\in H^1_{b,0}(\hold)$ are solutions of (see \cite{MR3535238})
\begin{align*}
\partial_{(\xa,\xb)}\mathcal{L}(\bfom,(\ua,\ub),(\pa,\pb))(\hat\xa,\hat\xb) = 0\quad \forall (\hat\xa,\hat\xb) \in H^1(\hold)\times H^1(\hold).
\end{align*}
This yields
\begin{align}
\label{58a}\int_\hold \sigma_{\bfom}^\transp \nabla \pa \cdot \nabla \hat\xa &= -\int_{\hold} (\ua - \ub)\hat\xa  +  \int_{\Gamma_a}(\sigma_{\bfom}^\transp \nabla\pa)\cdot\bn    \hat\xa\quad \mbox{ for all } \hat\xa\in H^1(\hold),\\
\label{58b}\int_\hold \sigma_{\bfom}^\transp \nabla \pb \cdot \nabla \hat\xb &= \int_{\hold} (\ua - \ub)\hat\xb  +  \int_{\Gamma_b}(\sigma_{\bfom}^\transp \nabla\pb)\cdot\bn  \hat\xb\quad \mbox{ for all } \hat\xb\in H^1(\hold).
\end{align}
Taking test functions  $\hat\xa\in  H^1_{a,0}(\hold)\subset H^1(\hold)$ and $\hat\xb\in H^1_{b,0}(\hold)\subset H^1(\hold)$ in \eqref{58a}-\eqref{58b}, we get
\eqref{57a}-\eqref{57b}.

Following the averaged adjoint method \cite{MR3535238},  we introduce the shape-Lagrangian $G$ using a reparameterization of $\mathcal{L}$:
\begin{align*}
& G(t,(\xa,\xb),(\ya,\yb)) := \mathcal{L}(\bfom_t,(\xa,\xb)\circ\Tt^{-1},(\ya,\yb)\circ\Tt^{-1})  \\
& = \frac{1}{2}\int_{\hold} (\xa^t -  \xb^t)^2 
+ \int_{\hold} \sigma_{\bfom_t} D\Tt^{-\transp}\circ\Tt^{-1}(\nabla \xa)\circ\Tt^{-1} \cdot D\Tt^{-\transp}\circ\Tt^{-1}(\nabla \ya)\circ\Tt^{-1} \\
&\quad - \int_\hold f_{\bfom_t} \ya^t - \int_{\Gamma_b} g\ya 
- \int_{\Gamma_a}(\sigma_{\bfom}^\transp \nabla\ya)\cdot\bn  (\xa - \meas)\\
&\quad + \int_{\hold} \sigma_{\bfom_t} D\Tt^{-\transp}\circ\Tt^{-1}(\nabla \xb)\circ\Tt^{-1} \cdot D\Tt^{-\transp}\circ\Tt^{-1}(\nabla \yb)\circ\Tt^{-1}\\
&\quad - \int_\hold f_{\bfom_t} \yb^t - \int_{\Gamma_a} g\yb
-  \int_{\Gamma_b}(\sigma_{\bfom}^\transp \nabla\yb)\cdot\bn  (\xb - \meas), 
\end{align*} 
with the notation $\xa^t := \xi \circ\Tt^{-1}$ and the similar notations for the other functions involved. 
Note that we have used $\Tt = \text{id}$ on $\partial\hold$, where $\text{id}$ denotes the identity, due to $\bta\in\C^1_{\partial\hold}(\overline{\hold},\R^d)$.
Proceeding with the change of variables $\bx\mapsto\Tt(\bx)$ inside the integrals and using again $\Tt = \text{id}$ on $\partial\hold$, we get
\begin{align*}
G(t,(\xa,\xb),(\ya,\yb)) 
& = \frac{1}{2}\int_{\hold} (\xa -  \xb)^2 \det(\Tt)
+ \int_{\hold} \mathds{M}(t)\nabla \xa \cdot \nabla \ya\\
&- \int_\hold \widetilde f(t) \ya \det(\Tt) 
- \int_{\Gamma_b} g\ya 
- \int_{\Gamma_a}(\sigma_{\bfom}^\transp \nabla\ya)\cdot\bn  (\xa - \meas)\\
& + \int_{\hold} \mathds{M}(t)\nabla \xb \cdot \nabla \yb
- \int_\hold \widetilde f(t) \yb \det(\Tt) 
- \int_{\Gamma_a} g\yb
-  \int_{\Gamma_b}(\sigma_{\bfom}^\transp \nabla\yb)\cdot\bn  (\xb - \meas),
\end{align*} 
where $\mathds{M}(t) := \det(\Tt) D\Tt^{-1}\widetilde\sigma(t) D\Tt^{-\transp}$, 
$\widetilde\sigma(t): =  \sum_{k\in\K} \sigma_k\circ\Tt \chi_{\Om_k}$, $\widetilde f(t): =  \sum_{k\in\K} f_k\circ\Tt \chi_{\Om_k}$.
We compute the derivatives at $t=0$:
\begin{align*}
\widetilde f'(0) & =  \sum_{k\in\K} \chi_{\Om_k}\nabla f_k \cdot\bta ,\\
\widetilde\sigma'(0) & = \sum_{k\in\K} \chi_{\Om_k} D\sigma_k\bta , \\
\mathds{M}'(0) & = \divv(\bta) \sigma_{\bfom} - D\bta \sigma_{\bfom} - \sigma_{\bfom} D\bta^\transp + \widetilde\sigma'(0). 
\end{align*}
Note that  $D\sigma_k:\hold\to \R^{d\times d\times d}$ is a third-order tensor since $\sigma_k$ are matrix-valued functions, and $D\sigma_k\bta$ is matrix-valued.
This yields 
\begin{align*}
d\mJ(\bfom)(\bta)  & = \partial_t G(0,(\ua,\ub),(\pa,\pb))\\
& = \frac{1}{2}\int_{\hold} (\ua -  \ub)^2 \divv(\bta)
+ \int_{\hold} \mathds{M}'(0)\nabla \ua \cdot \nabla \pa
- \int_\hold \widetilde f'(0) \pa  + f_{\bfom} \pa \divv(\bta) \\
&\quad + \int_{\hold} \mathds{M}'(0)\nabla \ub \cdot \nabla \pb
- \int_\hold \widetilde f'(0) \pb  + f_{\bfom} \pb \divv(\bta).
\end{align*}
Using tensor calculus we compute
\begin{align*}
\mathds{M}'(0)\nabla \ua \cdot \nabla \pa & =  \divv(\bta) \sigma_{\bfom}\nabla \ua \cdot \nabla \pa 
- D\bta \sigma_{\bfom}\nabla \ua \cdot \nabla \pa 
- \sigma_{\bfom} D\bta^\transp\nabla \ua \cdot \nabla \pa 
+ \widetilde\sigma'(0)\nabla \ua \cdot \nabla \pa \\
& =  (\sigma_{\bfom}\nabla \ua \cdot \nabla \pa)   I_d : D\bta
-  D\bta : (\nabla \pa \otimes \sigma_{\bfom}\nabla \ua )
-  D\bta : (\nabla \ua\otimes \sigma_{\bfom}^\transp\nabla \pa)\\
&\quad + \sum_{k\in\K} \chi_{\Om_k} D\sigma_k\bta \nabla \ua \cdot \nabla \pa\\
& =  D\bta :[(\sigma_{\bfom}\nabla \ua \cdot \nabla \pa)   I_d 
-  \nabla \pa \otimes \sigma_{\bfom}\nabla \ua 
- \nabla \ua\otimes \sigma_{\bfom}^\transp\nabla \pa]
 + \sum_{k\in\K} \chi_{\Om_k} D\sigma_k^\transp \nabla \ua \nabla \pa  \cdot \bta,
\end{align*}
where $D\sigma_k^\transp $ denotes the transpose of the third-order tensor $D\sigma_k$.
The other terms of $d\mJ(\bfom)(\bta)$ can be rearranged in a similar way to obtain \eqref{T:tensor_shape_deriv}.
\end{proof}
\begin{figure}
\begin{center}
\definecolor{preto}{RGB}{0,0,0}
\definecolor{vermelho}{RGB}{254,49,49}
\definecolor{azul_escuro}{RGB}{0,0,255}
\definecolor{cinza}{RGB}{195,195,195}
\definecolor{papel_amarelo}{RGB}{255, 249, 240}

\begin{tikzpicture}[scale=5]
\fill[papel_amarelo, opacity=0.4] (0,0) rectangle (1,1);
\draw[line width=0.7mm, preto](0,1) -- (1,1);
\draw[line width=0.7mm, vermelho](0,0) -- (0,1);	
\draw[line width=0.7mm, preto](0,0) -- (1,0);
\draw[line width=0.7mm, vermelho](1,0) -- (1,1);
\draw[line width=0.5mm, preto] (0.35,0.65) ellipse (0.2cm and 0.1cm) node[anchor= center] {$\Omega_1,\sigma_1$};
\fill[vermelho, opacity=0.2] (0.35,0.65) ellipse (0.2cm and 0.1cm);
\draw[line width=0.5mm, preto] (0.75,0.35) ellipse (0.15cm and 0.25cm) node[anchor= center] {$\Omega_2,\sigma_2$};
\fill[vermelho, opacity=0.4] (0.75,0.35) ellipse (0.15cm and 0.25cm);
\draw (0.35, 0.35) node[anchor= center] {$\Om_0,\sigma_0$};
\def\raio{0.015}
\def\transparenc{0.7}	
\def\h{0.1};
\draw (0.5, -0.08) node[anchor= center] {$\Gamma_b$};
\draw (0.5, 1.08) node[anchor= center] {$\Gamma_b$};
\draw[vermelho] (1.1, 0.5) node[anchor= center] {$\Gamma_a$};
\draw[vermelho] (-0.1, 0.5) node[anchor= center] {$\Gamma_a$};
\end{tikzpicture}
\end{center}
\caption{Example of partition  $\overline{\hold} = \overline{\Omega_0}\cup\overline{\Omega_1}\cup \overline{\Omega_2}$ and boundaries $\Gamma_a = \Gamma_{\rm{left}} \cup \Gamma_{\rm{right}}$, $\Gamma_b = \Gamma_{\rm{lower}}\cup \Gamma_{\rm{upper}}$. }\label{partition}
\end{figure}
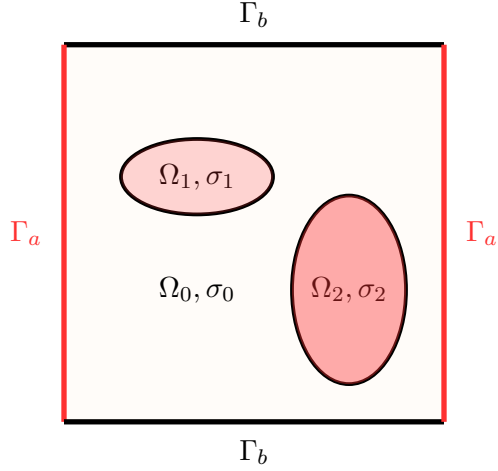
\begin{remark}
Formula \eqref{T:tensor_shape_deriv} generalizes \cite[Proposition 6.2]{MR3535238} in the case $\alpha_1 = 1$ and $\alpha_2 = 0$.
The result of \cite[Proposition 6.2]{MR3535238} can be recovered by taking two phases ($\K = \{0,1\}$) and 
$$\sigma_{\bfom} = \sigma_0\chi_{\Om_0} + \sigma_1\chi_{\Om_1},$$
where $\sigma_0,\sigma_1$ are multiples of the identity matrix.
\end{remark}

\subsection{Algorithm and numerical results}\label{sec:algo}
Without loss of generality, we take $\lsf_0 \equiv 0$ in the numerics.
The phases $\Om_k(\blsf(t))$ are defined as in \eqref{def:omk}. 
We consider the particular case $\ck=3$, $\K =\{0,1,2\}$, $d=2$ which corresponds to three phases in two dimensions.  
We choose  $\hold=(0,1)\times(0,1)$, $f_{\bfom}\equiv 0$ and $\sigma_{\bfom} = \sigma_0\chi_{\Om_0} + \sigma_2\chi_{\Om_1} + \sigma_2\chi_{\Om_2}$, where $\sigma_k = \sigma_{k,0}I_d$ for $k\in\K$, and $\sigma_{k,0}$ are known scalar values.
This corresponds to the isotropic EIT case, and \eqref{S1_1}-\eqref{S0} become in this case
\begin{align*}
S_1(\bfom) & = \left[\frac{1}{2}(\ua-\ub)^2 
+  \sigma_{\bfom}\nabla \ua\cdot\nabla \pa + \sigma_{\bfom}\nabla \ub\cdot\nabla \pb 
\right]  I_d
- 2\sigma_{\bfom} [\nabla \ua\odot \nabla \pa 
+ \nabla \ub\odot \nabla \pb],\\
\Sb_0(\bfom) & \equiv  0,  
\end{align*}
where $\nabla \ua\odot \nabla \pa := (\nabla \ua\otimes \nabla \pa+\nabla \pa\otimes \nabla \ua)/2$.

We use the software package FEniCS for the implementation; see \cite{fenics:book}.
The square $\hold$ is discretized using a triangular mesh with $255$ cells in both directions.
In our numerical experiments, we choose $\sigma_{0,0} = 1$, $\sigma_{1,0} = 3$, $\sigma_{2,0} = 15$.  
We choose  $\Gamma_a = \Gamma_{\rm{left}} \cup \Gamma_{\rm{right}}$ with 
$\Gamma_{\rm{left}} = \{0\}\times [0,1]$, $\Gamma_{\rm{right}}= \{1\}\times [0,1]$ and  $\Gamma_b = \Gamma_{\rm{lower}}\cup \Gamma_{\rm{upper}}$ with $\Gamma_{\rm{lower}} = [0,1]\times \{0\}$,  $\Gamma_{\rm{upper}} = [0,1]\times \{1\}$; see Figure \ref{partition}.
We also normalize the cost function \eqref{eit3.2} and the associated shape derivative by dividing them by the constant $\frac{1}{2}\int_{\hold} (\ua_1^{(0)} -  \ub_1^{(0)})^2$, where $\ua_1^{(0)}$,$\ub_1^{(0)}$ represent $\ua_1$ and $\ub_1$ computed for the initial partition $\bfom^{(0)}$.

Synthetic measurements $h_i$ are obtained by taking the trace on $\partial\hold$ of the solution
of \eqref{eq:eit} using the ground truth partition $\bfom^\star$,  $f_{\bfom^\star}\equiv 0$ and  currents $g_i$, $i = 1, \dots, m$. 
To simulate noisy EIT data, each measurement $h_i$
is corrupted by adding a normal Gaussian noise with mean zero and standard deviation $\delta \|h_i\|_{\infty}$, where $\delta$ is
a parameter. 
The noise level is computed as
\begin{align}
noise =& \frac{\sum_{i=1}^m  \|h_i - \tilde{h}_i\|_{L^2(\partial\hold)}}{\sum_{i=1}^m  \|h_i\|_{L^2(\partial\hold)}},
\end{align}
where $h_i$ and $\tilde{h}_i$ are respectively the noiseless and noisy measurements corresponding to the current $g_i$.

In the numerical tests, we use $I = 11$ measurements and define the currents in the following way:
\begin{align*}
g_1 &= 1\mbox{ on }\Gamma_{\rm{left}}\cup\Gamma_{\rm{right}} \mbox{ and } g_1 = -1\mbox{ on }\Gamma_{\rm{upper}}\cup\Gamma_{\rm{lower}}, \\ 
g_2 &= 1\mbox{ on }\Gamma_{\rm{left}}\cup\Gamma_{\rm{upper}} \mbox{ and } g_2 = -1\mbox{ on }\Gamma_{\rm{right}}\cup\Gamma_{\rm{lower}},\\
g_3 &= 1\mbox{ on }\Gamma_{\rm{left}}\cup\Gamma_{\rm{lower}} \mbox{ and } g_3 = -1\mbox{ on }\Gamma_{\rm{right}}\cup\Gamma_{\rm{upper}}.
\end{align*}
Then we choose 
\begin{align*}
g_4 &= \arctan(500(x_2 - 0.5))\mbox{ on }\Gamma_{\rm{left}} \mbox{ and } g_4=0 \mbox{ otherwise}, 
\end{align*}
which is used as an approximation of the function
\begin{align*}
g &= \frac{\pi}{2}\mbox{ on }\Gamma_{\rm{left}}\cap \{x_2 > 0.5\},\ g = -\frac{\pi}{2}\mbox{ on }\Gamma_{\rm{left}}\cap\{x_2 \leq 0.5\} \mbox{ and } g_4=0 \mbox{ otherwise}, 
\end{align*}
and $g_5,g_6,g_7$ are defined similarly as $g_4$ on $\Gamma_{\rm{right}}$, $\Gamma_{\rm{upper}}$, $\Gamma_{\rm{lower}}$, respectively. 
Then
\begin{align*}
g_8 &= \sin(4 \pi x_2)\mbox{ on }\Gamma_{\rm{left}}  \mbox{ and } g_8=0 \mbox{ otherwise},
\end{align*}
and $g_9,g_{10},g_{11}$ are defined in a similar way on $\Gamma_{\rm{right}}$, $\Gamma_{\rm{upper}}$, $\Gamma_{\rm{lower}}$, respectively.

In order to obtain a descent direction we solve  
\begin{equation}
\label{VP_2}
\mathcal{B}(\bta,\bxi) :=\int_\hold \alpha_1 D\bta : D\bxi +  \alpha_2 \bta\cdot \bxi + \int_{\partial\hold}\alpha_3 \bta\cdot \bxi =- d\mJ(\bfom)(\bxi) \mbox{  for all  } \bxi\in H^1(\hold)^2,
\end{equation}
with $\alpha_1 >0,\alpha_2 >0$ and $\alpha_3 >0$. 
For $d\mJ(\bfom)$ one can in principle either use the distributed expression \eqref{eq:first_order_tensor} or the boundary expression \eqref{158}, but we use the 
distributed expression \eqref{eq:first_order_tensor} which is convenient for implementation.
The solution $\bta$ of \eqref{VP_2} is defined on all of $\hold$ and is a descent direction since $\dJ(\bfom)(\bta) = -\mathcal{B}(\bta,\bta)< 0$ if $\bta\neq 0$.
In our experiment we used $\alpha_1 =0.2, \alpha_2 =0.8$ and $\alpha_3 =10^5$.
The role of the large coefficient $\alpha_3$ is to provide a relaxation of the Dirichlet boundary condition $\bta\in\C^1_{\partial\hold}(\overline{\hold},\R^d)$ so that slow tangential displacements can occur on $\partial\hold$, which allows to consider discontinuities of the conductivity up to the boundary $\partial\hold$.

We define a relative error measure for the reconstruction as (note that $E(\bfom)$ is a percentage)
$$ E(\bfom) := 100\times \frac{\displaystyle\int_\hold |\sigma_{\bfom} - \sigma_{\bfom^\star}|}{\displaystyle\int_\hold |\sigma_{\bfom}|}.$$

Numerical results are shown in Figures \ref{fig4} and \ref{fig5}. 
The ground truth conductivity $\sigma(\bfom^\star)$ is composed of 
a background with two low conductivity phases $\sigma_{0,0} = 1$ and $\sigma_{1,0} = 3$, separated by a curvy horizontal interface, and of
two inclusions of different sizes and higher conductivity $\sigma_{2,0} = 15$   (see the ground truth $\sigma(\bfom^\star)$  in Figure \ref{fig4}).
The goal is to reconstruct the shapes of the two conductivities and the location of the interface between the two low conductivity phases.
As can be seen in Figures \ref{fig4} and \ref{fig5}, the shapes of the two inclusions are well-reconstructed albeit slightly smoothed.
The interface between the two weak phases is well-reconstructed in the regions closer to the boundary, and less so in the center, as expected for this type of inverse problem. 

In Figure \ref{fig5} the sensitivity of the reconstruction with respect to noise is investigated. 
Numerical results corresponding to three different noise levels are compared. 
In the three cases, the reconstruction is able to capture the main geometric features of the ground truth.
The relative errors at the final iteration  corresponding to the noise levels $0\%,1.02\%$ and $2.03\%$ are given by $E(\bfom^{\text{rec}}) =5.72\%, 6.19\%$ and $6.81\%$, respectively, thus showing  that the method is robust with respect to noise.\\

\begin{figure}
\includegraphics[width=0.32\textwidth]{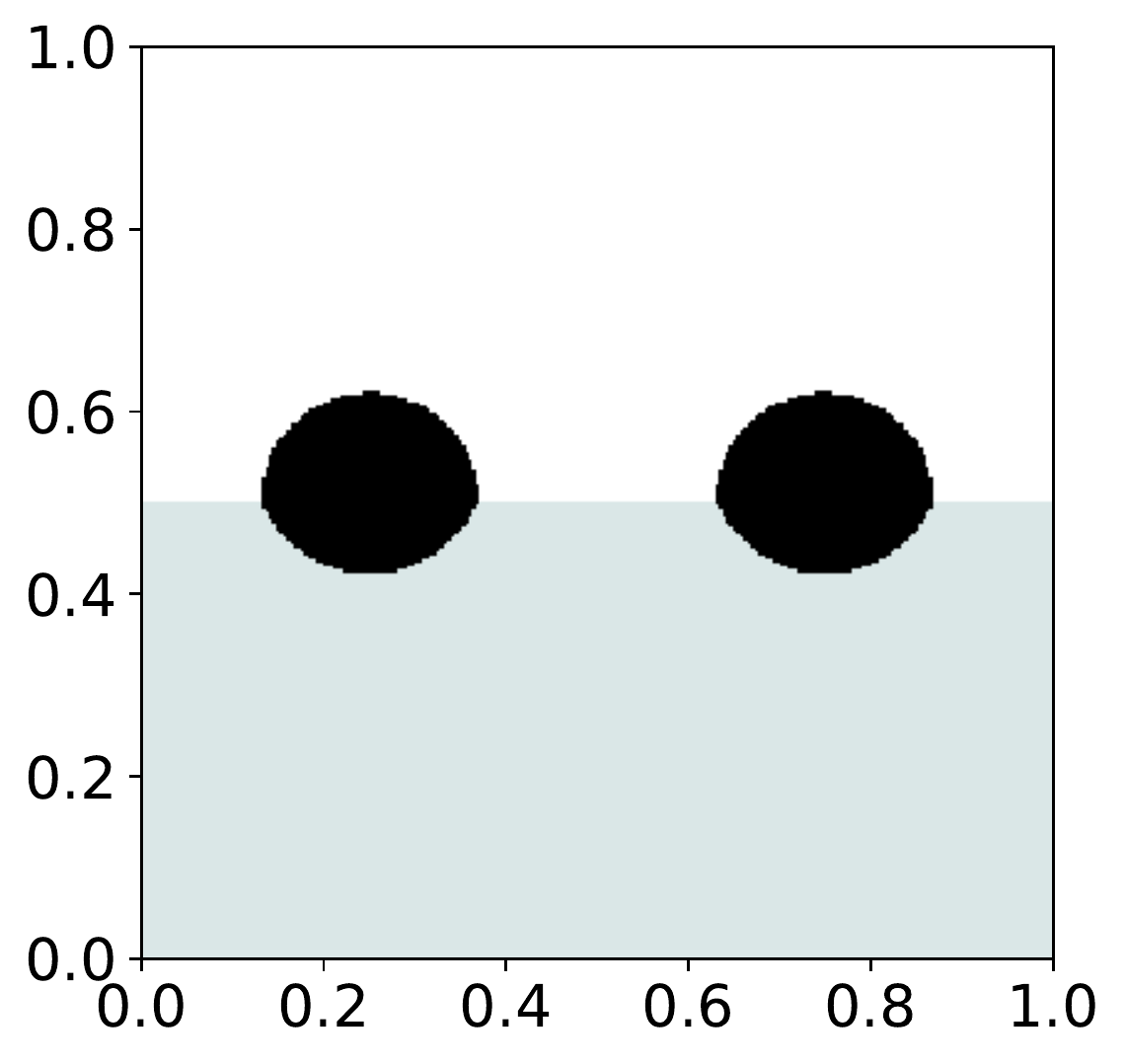}
\includegraphics[width=0.32\textwidth]{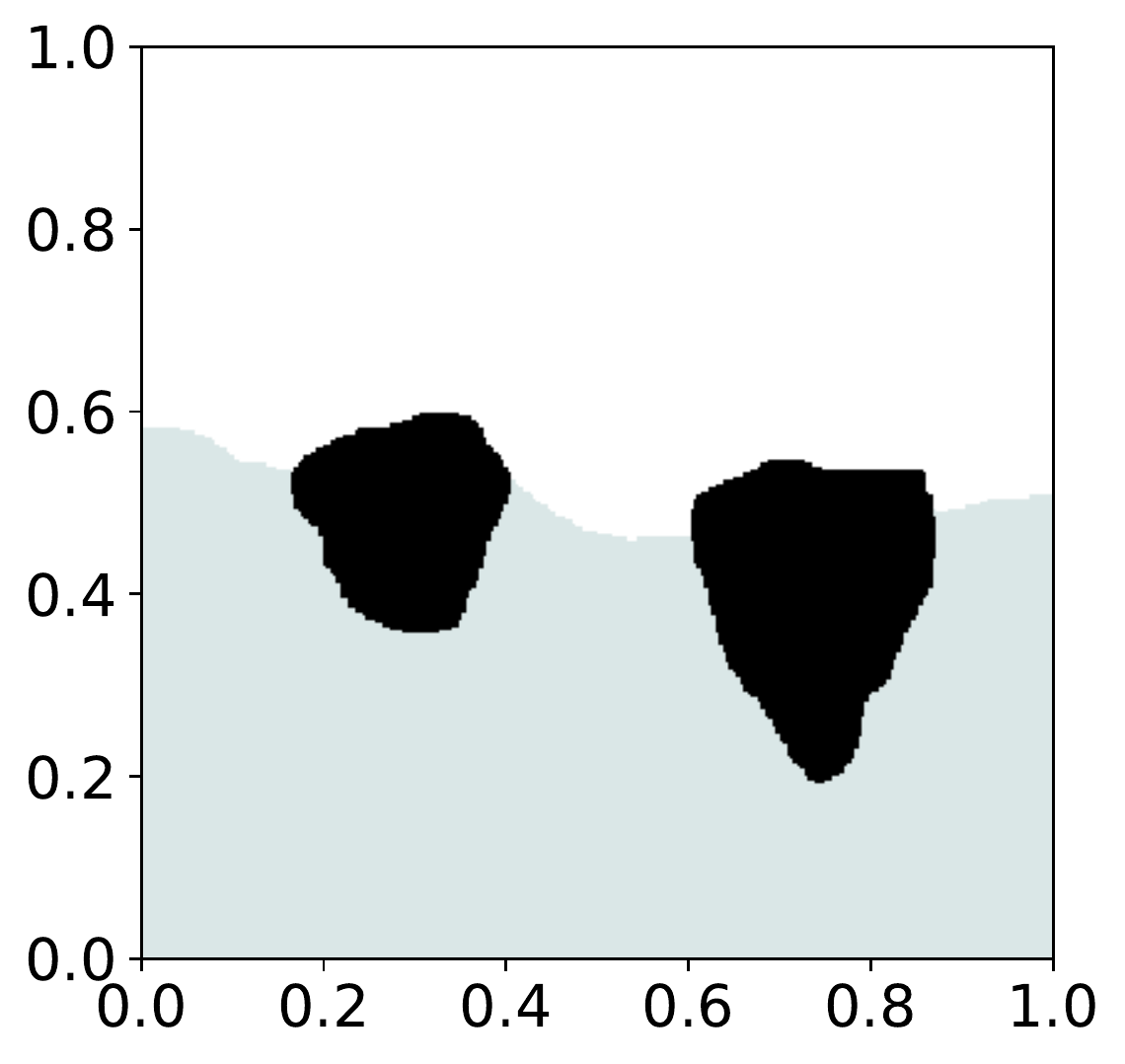}
\includegraphics[width=0.32\textwidth]{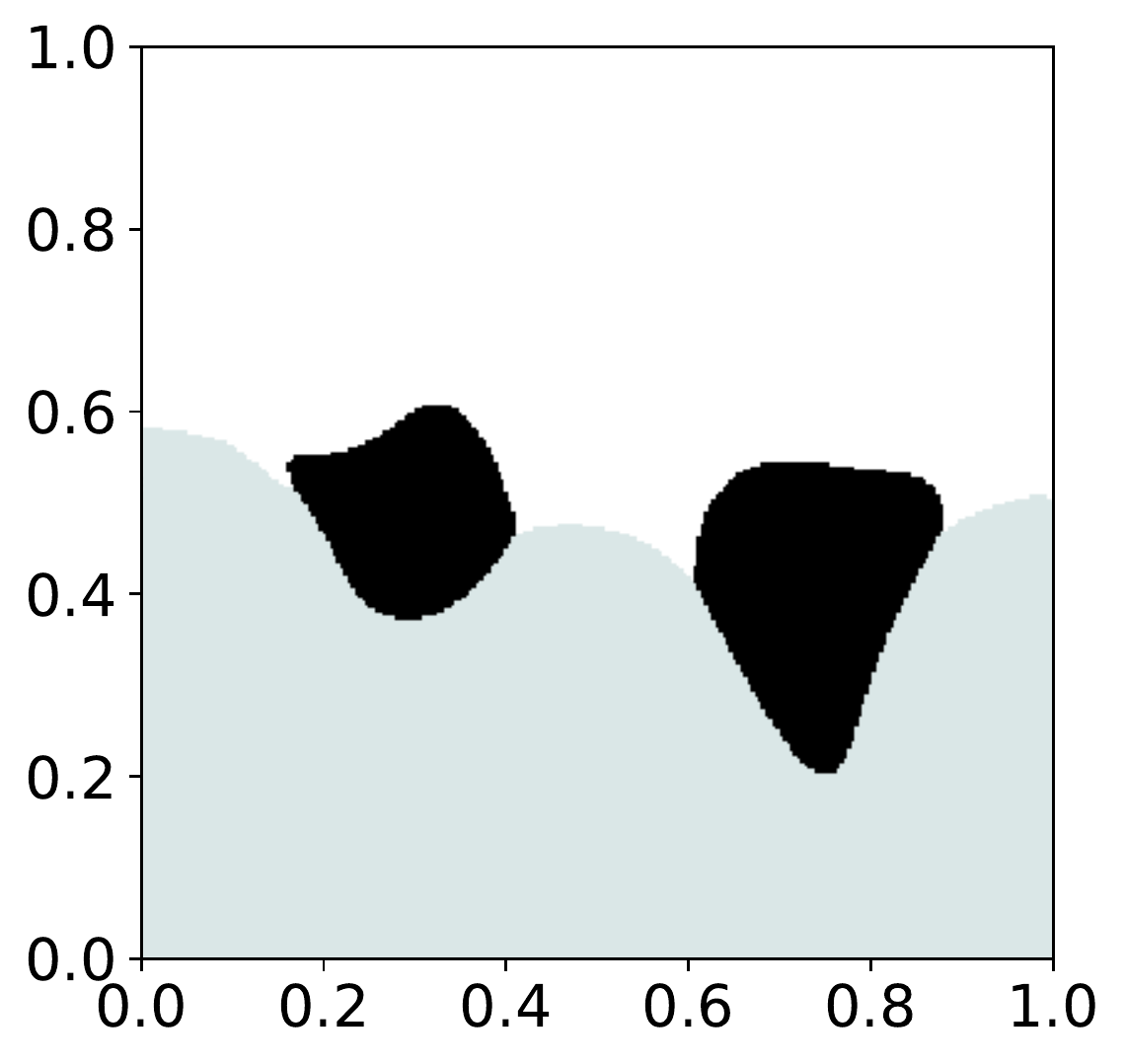}
\caption{Initialization (left), ground truth $\sigma(\bfom^\star)$ (center) and reconstruction $\sigma(\bfom^{\text{rec}})$ using $11$ boundary currents with $0\%$ noise  and $E(\bfom^{\text{rec}}) = 5.72\%$ reconstruction error  (right). The conductivity values are  $\sigma_{0,0} = 1$ (white), $\sigma_{1,0} = 3$ (light gray), $\sigma_{2,0} = 15$ (black).}\label{fig4}
\end{figure}
\begin{figure}
\includegraphics[width=0.45\textwidth]{true_conductivity.pdf}
\includegraphics[width=0.45\textwidth]{reconstruction.pdf}\\
\includegraphics[width=0.45\textwidth]{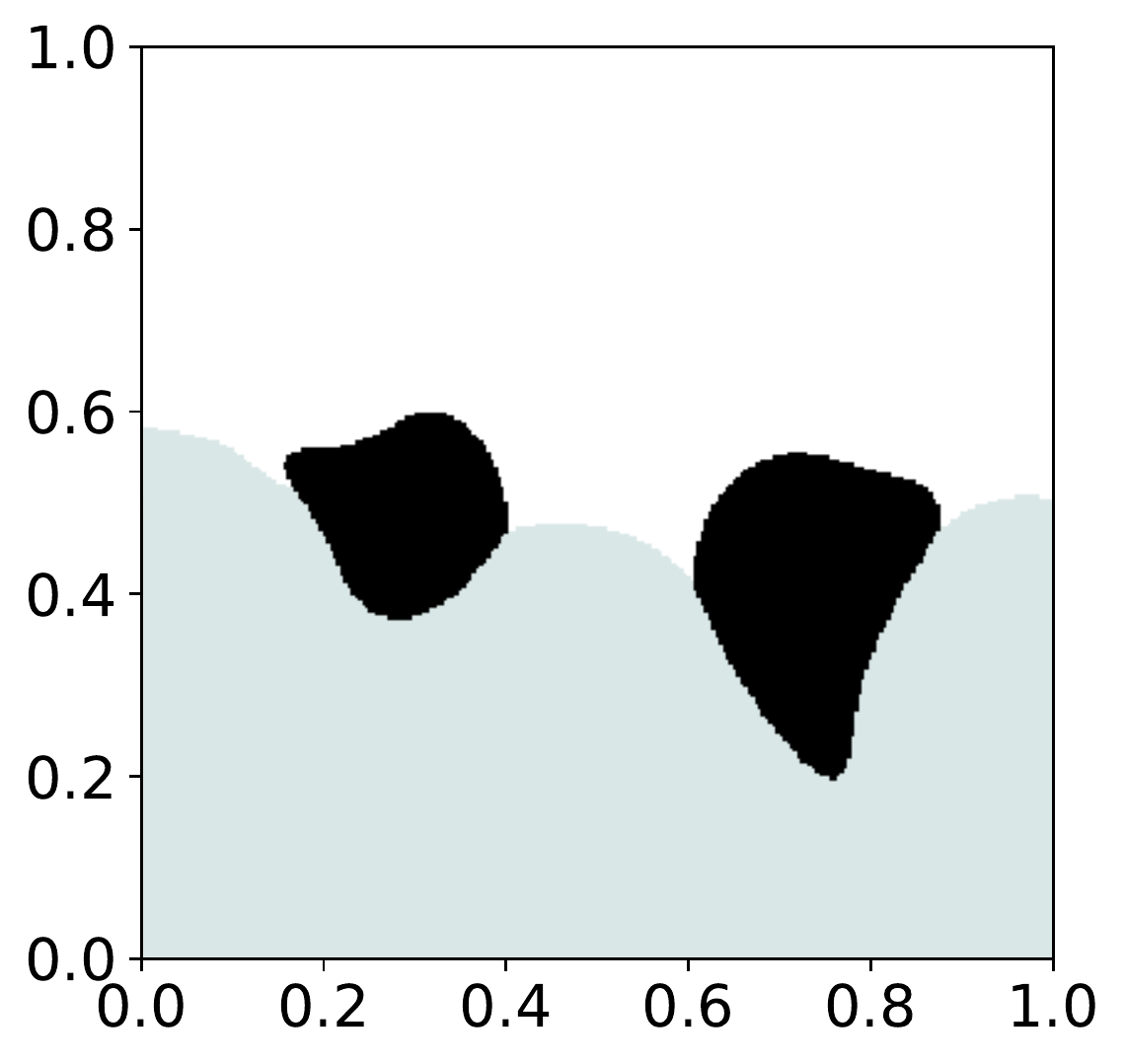}
\includegraphics[width=0.45\textwidth]{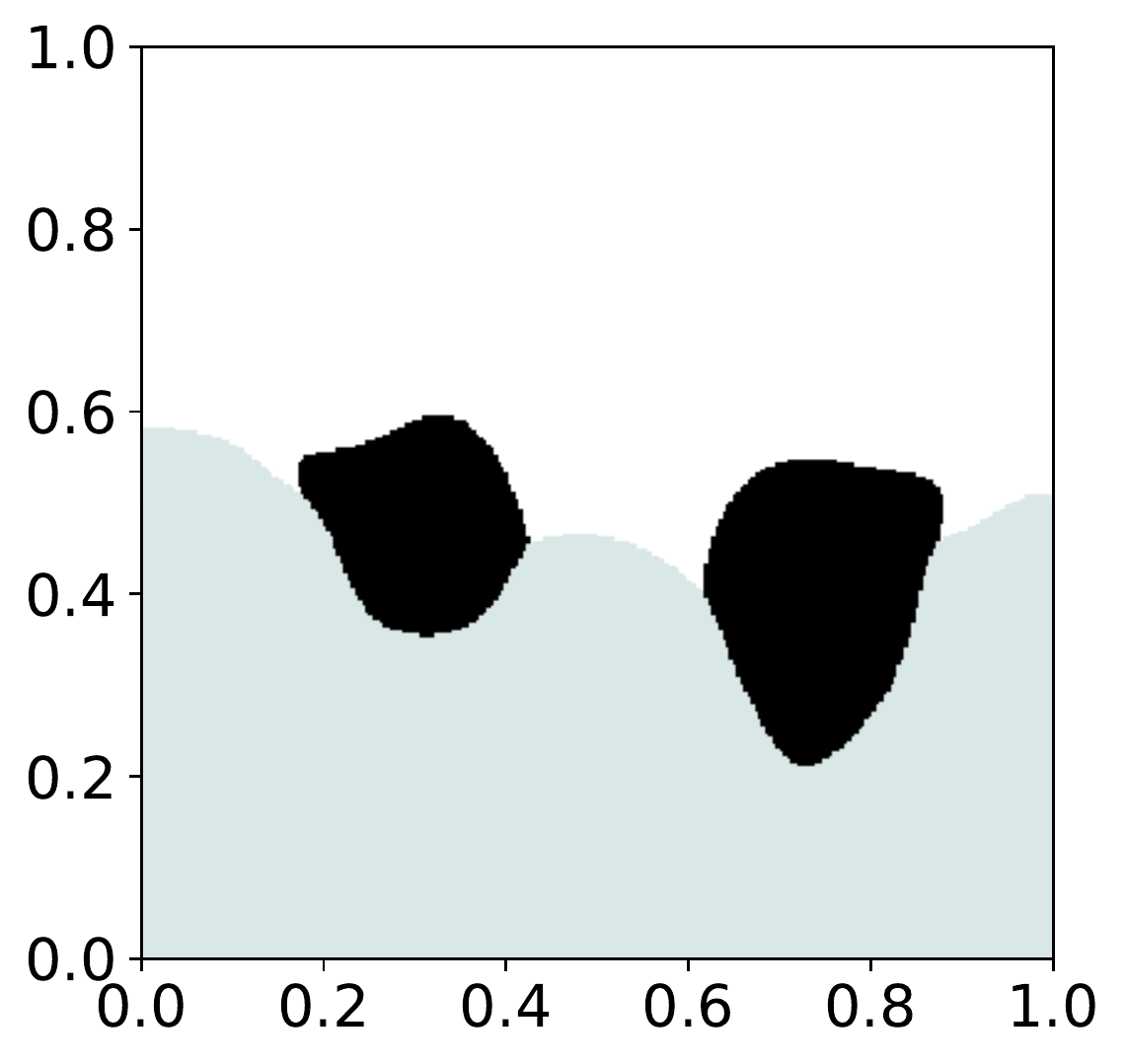}
\caption{Ground truth $\sigma(\bfom^\star)$ (top left), reconstructions  $\sigma(\bfom^{\text{rec}})$ using $11$ boundary currents with $0\%$ noise  and $E(\bfom^{\text{rec}}) =5.72\%$ error (top right), with $1.02\%$ noise  and $E(\bfom^{\text{rec}}) = 6.19\%$ error (bottom left), with $2.03\%$ noise and $E(\bfom^{\text{rec}}) =6.81\%$ error  (bottom right). The conductivity values are  $\sigma_{0,0} = 1$ (white), $\sigma_{1,0} = 3$ (light gray), $\sigma_{2,0} = 15$ (black). The initialization is shown in Figure \ref{fig4}(left).}\label{fig5}
\end{figure}

\subsection*{ Acknowledgments.} The author would like to thank Prof. James A. Sethian for the discussion about the tracking of interfaces in multiphase problems and  the inspiration to work on this topic. 

\subsection*{Funding}
The author gratefully acknowledges the support  of the Brazilian National Council for Scientific and Technological Development  (Conselho Nacional de Desenvolvimento Cient\'ifico e Tecnol\'ogico - CNPq) through the process: 408175/2018-4 ``Otimiza\c{c}\~ao de forma n\~ao suave e controle de problemas de fronteira livre'', and through the program  ``Bolsa de Produtividade em Pesquisa - PQ 2018'', process: 304258/2018-0.

\bibliographystyle{abbrv}
\bibliography{LEM}
\end{document}